\documentclass[11pt, a4paper]{article}

\usepackage{graphicx,url}
\usepackage{setspace}
\usepackage{enumerate,lineno,setspace,float}
\usepackage{amsmath,amsfonts,amssymb,amsthm}
\usepackage{geometry}
\usepackage{fancyhdr}
\usepackage{amssymb,amsmath}
\usepackage{latexsym}
\usepackage{authblk}

\newtheorem{Theorem}{Theorem}[section]

\newtheorem{Corollary}[Theorem]{Corollary}
\newtheorem{Lemma}[Theorem]{Lemma}

\theoremstyle{definition}
\newtheorem{Definition}{Definition}[section]
\newtheorem{Remark}{Remark}[section]

\usepackage{color}

\definecolor{Blue}{rgb}{0,0,1}
\definecolor{DarkGreen}{rgb}{0,0.6,0}
\definecolor{Red}{rgb}{1,0,0}
\definecolor{Orange}{rgb}{1,0.5,0}

\long\def\delete#1{}

\usepackage{xcolor}
\usepackage[normalem]{ulem}

\input amssym.def
\input amssym.tex

\newcommand{\be}{\begin{equation}}
\newcommand{\ee}{\end{equation}}
\newcommand{\bea}{\begin{eqnarray}}
\newcommand{\eea}{\end{eqnarray}}
\newcommand{\bean}{\begin{eqnarray*}}
\newcommand{\eean}{\end{eqnarray*}}

\def\non{\nonumber}

\def\diam{{\rm diam}}
\def\span{{\rm span}}

\def\rn{{\rm rn}}

\def\d{\delta}

\def\ve{\varepsilon}

\def\mt{\mathcal}

\def\({\left(}
\def\){\right)}
\def\[{\left[}
\def\]{\right]}

\begin{document}

\title{Radio labelling of two-branch trees}

\author[a]{Devsi Bantva\thanks{E-mail: \texttt{devsi.bantva@gmail.com}}}
\author[b]{Samir Vaidya\thanks{E-mail: \texttt{samirkvaidya@yahoo.co.in}}}
\author[c]{Sanming Zhou\thanks{E-mail: \texttt{sanming@unimelb.edu.au}}}
\affil[a]{{\small Department of Mathematics, Lukhdhirji Engineering College, Morvi 363 642, Gujarat, India}}
\affil[b]{{\small Department of Mathematics, Saurashtra University, Rajkot 360 005, Gujarat, India}}
\affil[c]{{\small School of Mathematics and Statistics, The University of Melbourne, Parkville, VIC 3010, Australia}}


\date{}
\openup 0.5\jot
\maketitle

\begin{abstract}
A radio labelling of a graph $G$ is a mapping $f : V(G) \rightarrow \{0, 1, 2,\ldots\}$ such that $|f(u)-f(v)|\geq \diam(G) + 1 - d(u,v)$ for every pair of distinct vertices $u,v$ of $G$, where $\diam(G)$ is the diameter of $G$ and $d(u,v)$ is the distance between $u$ and $v$ in $G$. The radio number $\rn(G)$ of $G$ is the smallest integer $k$ such that $G$ admits a radio labelling $f$ with $\max\{f(v):v \in V(G)\} = k$. The weight of a tree $T$ from a vertex $v \in V(T)$ is the sum of the distances in $T$ from $v$ to all other vertices, and a vertex of $T$ achieving the minimum weight is called a weight center of $T$. It is known that any tree has one or two weight centers. A tree is called a two-branch tree if the removal of all its weight centers results in a forest with exactly two components. In this paper we obtain a sharp lower bound for the radio number of two-branch trees which improves a known lower bound for general trees. We also give a necessary and sufficient condition for this improved lower bound to be achieved. Using these results, we determine the radio number of two families of level-wise regular two-branch trees.


\end{abstract}

\section{Introduction}\label{intro}

The channel assignment problem is the problem of assigning a channel to each transmitter in a radio network such that a set of constraints, usually determined by the geographic proximity of the transmitters, are satisfied and the span is minimized. In \cite{Hale}, Hale modelled this problem as an optimal labelling problem for graphs, in which vertices represent transmitters and two vertices are adjacent if the corresponding transmitters are close to each other. Initially, only two levels of interference were considered, leading to the following concept \cite{Griggs}: An \emph{$L(2,1)$-labelling} of a graph $G$ is a function $f : V(G) \rightarrow \{0, 1, 2, \ldots\}$ such that $|f(u)-f(v)| \geq 2$ if $d(u,v) = 1$ and $|f(u)-f(v)| \geq 1$ if $d(u,v) = 2$, where $V(G)$ is the vertex set of $G$ and $d(u,v)$ is the distance between $u$ and $v$ in $G$. The span of $f$ is defined as $\max\{|f(u)-f(v)|: u, v \in V(G)\}$, and the \emph{$\lambda$-number} (or the $\lambda_{2,1}$-number) of $G$ is the minimum span of an $L(2,1)$-labelling of $G$. The $L(2,1)$-labelling problem has been studied extensively in the past more than two decades, as one can find in the survey articles \cite{Calamoneri, Yeh1}.

In \cite{Chartrand1,Chartrand2}, Chartrand \emph{et al.} introduced the following radio labelling problem by considering all levels of interference up to the diameter of the graph, where the diameter of a graph $G$ is defined as $\diam(G) = \max\{d(u,v) : u, v \in V(G)\}$.

\begin{Definition}
A \emph{radio labelling} of a graph $G$ is a mapping $f: V(G) \rightarrow \{0, 1, 2, \ldots\}$ such that for every pair of distinct vertices $u, v$ of $G$,
$$
|f(u)-f(v)| \geq \diam(G) + 1 - d(u,v).
$$
The integer $f(u)$ is called the \emph{label} of $u$ under $f$, and the \emph{span} of $f$ is defined as
$$
\span(f) = \max\{|f(u)-f(v)|: u, v \in V(G)\}.
$$
The \emph{radio number} of $G$ is defined as
$$
\rn(G) = \min_{f} \span(f)
$$
with minimum taken over all radio labellings $f$ of $G$. A radio labelling $f$ of $G$ is called \emph{optimal} if $\span(f) = \rn(G)$.
\end{Definition}

Without loss of generality we may always assume that any radio labelling assigns $0$ to some vertex. With this convention the span of a radio labelling is equal to the maximum label used. Observe that any radio labelling should assign different labels to distinct vertices. So a radio labelling $f$ of a graph $G$ with order $p$ induces the linear order
\be
\label{eqn:ord}
OV_{f}(G): u_{0}, u_{1}, \ldots, u_{p-1}
\ee
of the vertices of $G$ such that
\be
\label{eq:spf}
0 = f(u_{0}) < f(u_{1}) < \ldots < f(u_{p-1}) = \span(f).
\ee

In general, it is challenging to determine the exact value of the radio number of a graph. As such much attention has been paid to special families of graphs such as trees. It turns out that even for some special families of trees the problem can be difficult. For example, the radio number of paths was determined by Liu and Zhu in \cite{Liu}, and even in this innocent-looking case the problem is nontrivial. In \cite{Benson}, Benson \emph{et al.} determined the radio number of all graphs of order $n \ge 2$ and diameter $n-2$. In \cite{Bantva3}, Bantva determined the radio number of some trees obtained by applying a graph operation on given trees. In \cite[Theorem 3]{Daphne1}, Liu gave a lower bound for the radio number of trees and a necessary and sufficient condition for this bound to be achieved. She also presented a class of spiders achieving this lower bound, where a spider is a tree with at most one vertex of degree greater than two. In \cite[Theorems 1 and 2]{Li}, Li \emph{et al.} determined the radio number of complete $m$-ary trees of height $k$, for $m \geq 2$ and $k \geq 1$. In \cite{Tuza}, Hal\'{a}sz and Tuza gave a lower bound for the radio number of level-wise regular trees and proved further that this bound is tight under some additional conditions, generalizing a result in \cite{Li}. In \cite[Theorem 3.2]{Bantva2}, the authors of the present paper gave a necessary and sufficient condition for the lower bound in \cite[Theorem 3]{Daphne1} to be tight along with two sufficient conditions for achieving this lower bound. Using these results, we also determined in \cite{Bantva2} the radio number of three families of trees.

Observe that in many cases the lower bound in \cite[Theorem 3]{Daphne1} can be far from being optimal, though it can be achieved in some other cases. This motivated us to study when this bound can be improved. In this paper we investigate this problem with a focus on a large family of trees called two-branch trees. (A tree is called a two-branch tree if the deletion of its weight centers results in a forest with exactly two components, where a vertex in a tree is called a weight center if it has the least total distance to other vertices.) We first obtain a necessary condition for a two-branch tree to attain the lower bound in \cite[Theorem 3]{Daphne1} (see Theorem \ref{thm:nlb}). We then give an improved lower bound for two-branch trees and obtain further a necessary and sufficient condition for our new bound to be achieved (see Theorem \ref{thm:lower}). We show that our improved lower bound is sharp and can be significantly larger than the one in \cite[Theorem 3]{Daphne1} for some two-branch trees. This is achieved by constructing a family of two-branch trees for which the gap between the two bounds can be arbitrarily large (see Theorem \ref{thm:cat}). Our improved lower bound enables us to determine the radio number of two large families of level-wise regular two-branch trees (see Theorems \ref{thm:level} and \ref{thm:ds}), and in a special case we recover the above-mentioned result of Li \emph{et al.} \cite[Theorem 1]{Li} on the radio number of complete binary trees (see Corollary \ref{coro:Li}). We also notice that the above-mentioned result of Liu and Zhu \cite{Liu} for paths can be obtained by using our lower bound and related condition ensuring its tightness.

The rest of this paper is structured as follows. In Section \ref{prel}, we will introduce terminology and notation and recall a few results from \cite{Bantva2} and \cite{Daphne1}. In Section \ref{main}, we will prove Theorems \ref{thm:nlb}, \ref{thm:lower} and \ref{thm:cat}. As we will see in this section, the proof of Theorem \ref{thm:lower} requires a series of technical lemmas, and the proof of Theorem \ref{thm:cat} is reduced to showing the existence of a radio labelling achieving the lower bound in Theorem \ref{thm:lower}. Using the lower bound in Theorem \ref{thm:lower}, the proofs of Theorems \ref{thm:level} and \ref{thm:ds} will be given in Section \ref{level}. During the preparation of this paper (a project that we started in 2018) we became aware of the recent paper \cite{Daphne2}. By coincidence, Theorems \ref{thm:nlb} and \ref{thm:lower} partially answer a question in \cite{Daphne2}, and Theorems \ref{thm:level} and \ref{thm:ds} are closely related to \cite[Theorem 12]{Daphne2}. There are also overlaps between Theorem \ref{thm:lower} and \cite[Theorems 10 and 15]{Daphne2}, but neither of them implies the other. In Section \ref{sec:rem}, we conclude the paper with five remarks to explain these relations between our results and questions and results in \cite{Daphne2}.

\section{Preliminaries}
\label{prel}

We follow \cite{West} for graph-theoretic terminology and notation. Let $T$ be a tree. As usual, for $S \subseteq V(T)$, let $N(S) = \cup_{v \in S}N(v)$, where $N(v)$ is the neighbourhood of $v$ in $T$. A vertex $v \in V(T)$ is a \emph{leaf} of $T$ if it has degree one and is an \emph{internal vertex} otherwise. In \cite{Daphne1}, the \emph{weight} of $T$ from $v \in V(T)$ is defined as
$$
w_{T}(v) = \sum_{u \in V(T)}d(u,v)
$$
and the weight of $T$ is defined as
$$
w(T) = \min\{w_{T}(v) : v \in V(T)\}.
$$
A vertex $v \in V(T)$ is a \emph{weight center} of $T$ if $w_{T}(v) = w(T)$ \cite{Daphne1}. Denote by $W(T)$ the set of weight centers of $T$.
In \cite{Daphne1}, it was proved that $W(T)$ consists of either one vertex or two adjacent vertices, and moreover the latter occurs, say, $W(T) = \{w, w'\}$, if and only if $ww'$ is an edge of $T$ and $T-ww'$ consists of two components with the same number of vertices. In particular, we have $|W(T)| \in \{1, 2\}$.
Define
\begin{equation*}
\ve(T) = 2 - |W(T)| \in \{0, 1\}.
\end{equation*}

A vertex $u$ is called an \emph{ancestor} of a vertex $v$, or $v$ is a \emph{descendent} of $u$, if $u$ is on the unique path joining a weight center and $v$ \cite{Bantva2, Daphne1}. Let $u \in V(T) \setminus W(T)$ be a vertex adjacent to a weight center $x$. The subtree of $T$ induced by $u$ and all its descendants is called the \emph{branch} of $T$ at $u$. So the branches of $T$ are exactly the components of $T - W(T)$. Two vertices $u, v$ of $T$ are said to be in \emph{different branches} if the path between them traverses only one weight center and in \emph{opposite branches} if the path joining them traverses two weight centers. We view $T$ as rooted at its weight center $W(T)$: if $W(T) = \{w\}$, then $T$ is rooted at $w$; if $W(T) = \{w,w'\}$, then $T$ is rooted at $w$ and $w'$ in the sense that both $w$ and $w'$ are at level 0. In either case, for each $u \in V(T)$, define
$$
L(u) = \mbox{min}\{d(u,x) : x \in W(T)\}
$$
to indicate the \emph{level} of $u$ in $T$, and call
$$
L(T) = \sum_{u \in V(T)} L(u)
$$
the \emph{total level} of $T$.

For any $u, v \in V(T)$, define
\begin{equation}
\label{eq:phi}
\phi(u,v) = \mbox{max}\{L(x) : x \mbox{ is a common ancestor of $u$ and $v$}\}
\end{equation}
and
\begin{equation}
\label{eq:delta}
\delta(u,v) =
\begin{cases}
1, & \mbox{if $|W(T)| = 2$ and the $(u, v)$-path in $T$ contains both weight centers} \\
0, & \mbox{otherwise}.
\end{cases}
\end{equation}

\begin{Lemma}[{\cite[Lemma 2.1]{Bantva2}}]\label{lem4} Let $T$ be a tree with diameter $d \geq 2$. Then for any $u, v \in V(T)$ the following hold:
\begin{enumerate}[\rm (a)]
  \item $\phi(u,v) \geq 0$;
  \item $\phi(u,v) = 0$ if and only if $u$ and $v$ are in different or opposite branches;
  \item $\delta(u,v) = 1$ if and only if $T$ has two weight centers and $u$ and $v$ are in opposite branches;
  \item the distance $d(u,v)$ in $T$ between $u$ and $v$ can be expressed as
  \begin{equation}\label{eqn:dist}
  d(u,v) = L(u) + L(v) - 2\phi(u,v)+\delta(u,v).
  \end{equation}
\end{enumerate}
\end{Lemma}

\begin{Lemma}[{\cite[Theorem 3]{Daphne1}; see also \cite[Lemma 3.1]{Bantva2}}]
\label{lem:ddb}
Let $T$ be a tree with order $p$ and diameter $d \geq 2$. Then
\be
\label{eqn:lb}
\rn(T) \geq (p-1)(d+\ve(T))-2L(T)+\ve(T).
\ee
\end{Lemma}

\begin{Theorem}[{\cite[Theorem 3.2]{Bantva2}}]
\label{thm:ddb}
Let $T$ be a tree with order $p$ and diameter $d \ge 2$. Then
$$
\rn(T) = (p-1)(d+\ve(T)) - 2 L(T) + \ve(T)
$$
holds if and only if there exists a linear order $u_0, u_1, \ldots, u_{p-1}$ of the vertices of $T$ such that
\begin{enumerate}[\rm (a)]
\item $u_0 = w$ and $u_{p-1} \in N(w)$ when $W(T) = \{w\}$, and $\{u_0, u_{p-1}\} = \{w, w'\}$ when $W(T) = \{w, w'\}$;
\item the distance $d(u_{i}, u_{j})$ between $u_i$ and $u_j$ in $T$ satisfies
\be
\label{eq:dij}
d(u_{i}, u_{j}) \ge \sum_{t = i}^{j-1} \{L(u_t)+L(u_{t+1}) - (d+\ve(T))\} + (d+1),\;\, 0 \le i < j \le p-1.
\ee
\end{enumerate}
Moreover, under this condition the mapping $f$ defined by
\be
\label{f0}
f(u_{0}) = 0
\ee
\be\label{f1}
f(u_{i+1}) = f(u_{i}) - (L(u_{i}) + L(u_{i+1})) + (d + \ve(T)),\;\, 0 \leq i \leq p-2
\ee
is an optimal radio labelling of $T$.
\end{Theorem}

Recall from \eqref{eqn:ord} that, given a tree $T$ and a radio labelling $f$ of $T$, $OV_{f}(T): u_{0},u_{1}, \ldots, u_{p-1}$ denotes the linear order of the vertices of $T$ induced by $f$. Define
\be
\label{def:jvp}
J_{f}(u_{i}, u_{i+1}) = (f(u_{i+1}) - f(u_{i})) + d(u_{i}, u_{i+1}) - (d + 1),\ 0 \leq i \leq p-2.
\ee
Since $f$ is a radio labelling, we have $J_{f}(u_{i},u_{i+1}) \geq 0$ for each $i$. In general, for a segment $U = \{u_i, \ldots, u_j\}$ of $OV_{f}(T)$, where $0 \le i < j \le p-1$, we define
\bea
J_{f}(u_i, \ldots, u_j)
& = & \sum_{t = i}^{j-1} J_{f}(u_{t},u_{t+1}) \non \\
& = & (f(u_{j}) - f(u_{i})) + \sum_{t = i}^{j-1} \{d(u_{t}, u_{t+1}) - (d + 1)\}. \label{eq:jij}
\eea
Write $J_{f}(U) =  J_{f}(u_i, \ldots, u_j)$ for short. Define
\begin{equation}
\label{eq:jfT}
J_{f}(T) = J_{f}(u_{0}, \ldots, u_{p-1}).
\end{equation}
Then $J_{f}(T) = \span(f) + \sum_{t = 0}^{p-2} \{d(u_{t}, u_{t+1}) - (d + 1)\}$ by  \eqref{eq:spf} and \eqref{eq:jij}. In view of \eqref{eqn:dist}, this can be rewritten as
\bea
\span(f) & = & J_{f}(T) + \sum_{t=0}^{p-2} \{-(L(u_{t}) + L(u_{t+1})) + 2 \phi(u_{t},u_{t+1}) - \delta(u_{t},u_{t+1}) + (d + 1)\}\non \\
& = & (p-1)(d+1) - 2 L(T) + L(u_{0}) + L(u_{p-1}) + \sigma(f), \label{eqn:sumup}
\eea
where we set
\bea
\sigma(f) & = & J_{f}(T) + \sum_{t=0}^{p-2}\(2\phi(u_{t},u_{t+1}) - \delta(u_{t},u_{t+1})\) \non \\
& = & \sum_{t=0}^{p-2}\(J_{f}(u_{t},u_{t+1}) + 2\phi(u_{t},u_{t+1}) - \delta(u_{t},u_{t+1})\). \label{eq:sigf}
\eea
Since $\phi(u_{t},u_{t+1}) \geq 0$ and $J_{f}(u_{t},u_{t+1}) \geq 0$ for each $t$ and $\delta(u_{t},u_{t+1}) = 0$ when $|W(T)| = 1$ and $\delta(u_{t},u_{t+1}) \in \{0, 1\}$ when $|W(T)| = 2$, we have
\begin{equation}
\label{eq:sigf12}
\sigma(f) \ge
\begin{cases}
0, & \mbox{ if } |W(T)| = 1 \\
-(p-1), & \mbox{ if } |W(T)| = 2.
\end{cases}
\end{equation}

In the rest of this section we introduce a few definitions pertaining to the family of trees to be studied in this paper.

\begin{Definition}
\label{def:two-br-tree}
A tree with exactly two branches is called a \emph{two-branch tree}. In other words, a tree $T$ is called a two-branch tree if $T - W(T)$ has exactly two components.
\end{Definition}

Since for any weight center $w$ of a tree $T$, each component of $T-w$ contains at most $|V(T)|/2$ vertices (see \cite{Daphne1}), we have the following lemma.

\begin{Lemma}
\label{lem3}
Let $T$ be a two-branch tree with branches $T_1$ and $T_2$. Then $|V(T_1)|-|V(T_2)| = \pm 1$ if $|W(T)|=1$ and $|V(T)|$ is even, and $|V(T_1)| = |V(T_2)|$ otherwise.
\end{Lemma}

Let $T$ be a two-branch tree with $d = \diam(T)$. Define
$$
S(T) = \begin{cases}
\{u : u \in V(T), L(u) \geq \lceil d/2 \rceil\}, & \text{ if } |W(T)| = 1 \\
\{u : u \in V(T), L(u) \geq \lfloor d/2 \rfloor\}, & \text{ if } |W(T)| = 2
\end{cases}
$$
and call the vertices in $S(T)$ the \emph{remote vertices} of $T$. Obviously, if $|W(T)| = 1$, then $|S(T)| \geq |W(T)|$. However, we may have $|S(T)| < |W(T)|$ for some two-branch trees $T$ with $|W(T)| = 2$. Set
\begin{equation}
\label{eq:xit}
\xi(T) =
\begin{cases}
\lfloor |S(T)|/2 \rfloor, & \mbox{ if } |W(T)| = 1 \\
\max\{0, |S(T)|-2\}, & \mbox{ if } |W(T)| = 2.
\end{cases}
\end{equation}
Let
$$
OV(T): u_0,u_1,\ldots,u_{p-1}
$$
be a linear order of the vertices of $T$. A segment $u_i, \ldots, u_{j}$ of $OV(T)$ with \emph{length} $j-i+1$, where $0 \le i \le j \le p-1$, is called an \emph{interval of remote vertices} with respect to $OV(T)$ if each $u_t$ for $i \le t \le j$ is a remote vertex of $T$; such an interval is \emph{maximal} if neither $u_{i-1}$ (when $i > 0$) nor $u_{j+1}$ (when $j < p-1$) is a remote vertex of $T$. The linear order $OV(T)$ is called \emph{feasible} if all maximal intervals of remote vertices of $T$ have even lengths, except one when $|S(T)|$ is odd. The linear order $OV(T)$ is called \emph{admissible} if the following conditions are satisfied: (i) for each weight center $w$ of $T$, the unique vertex immediately before $w$ (if it exists) is a remote vertex, and the unique vertex immediately after $w$ (if it exists) is also a remote vertex; (ii) all maximal intervals of the remaining remote vertices of $T$ (that is, those not immediately before or after a weight center) have even lengths, except one when the number of remaining vertices is odd. (More specifically, condition (i) can be stated as follows: If $W(T) = \{u_0\}$, then $u_1 \in S(T)$; if $W(T) = \{u_{p-1}\}$, then $u_{p-2} \in S(T)$; if $W(T) = \{u_i\}$ for some $0 < i < p-1$, then $u_{i-1}, u_{i+1} \in S(T)$; if $W(T) = \{u_0, u_{p-1}\}$, then $u_1, u_{p-2} \in S(T)$; if $W(T) = \{u_0, u_{i}\}$ for some $i < p-1$, then $i > 2$ and $u_1,u_{i-1},u_{i+1} \in S(T)$; if $W(T) = \{u_i, u_{p-1}\}$ for some $i > 0$, then $i < p-3$ and $u_{i-1}, u_{i+1}, u_{p-2} \in S(T)$; and if $W(T) = \{u_i, u_j\}$ for some $0 < i < j < p-1$, then $j > i+2$ and $u_{i-1},u_{i+1},u_{j-1},u_{j+1}  \in S(T)$.)

\section{Radio number of two-branch trees}
\label{main}

\subsection{Results}

The lower bound in \eqref{eqn:lb} may not be tight for a general tree, though it can be achieved as shown in Theorem \ref{thm:ddb}. On the other hand, it is not straightforward to use Theorem \ref{thm:ddb} to check whether the lower bound in \eqref{eqn:lb} is tight for a given tree. Our first result gives a simple condition under which the lower bound in \eqref{eqn:lb} is slack for a two-branch tree.

\begin{Theorem}
\label{thm:nlb}
Let $T$ be a two-branch tree with order $p$ and diameter $d \geq 2$. If $|S(T)| > |W(T)|$, then
$$
\rn(T) > (p-1)(d+\ve(T))-2L(T)+\ve(T).
$$
\end{Theorem}

Our second result in this section is the theorem below. It shows that for two-branch trees the lower bound in \eqref{eqn:lb} can be improved by $\xi(T)$, where $\xi(T)$ is as defined in \eqref{eq:xit}. This theorem also gives a necessary and sufficient condition for our improved lower bound to be tight.

\begin{Theorem}
\label{thm:lower}
Let $T$ be a two-branch tree with order $p$ and diameter $d \geq 2$. Then
\be
\label{rn:lower}
\rn(T) \geq (p-1)(d+\ve(T)) - 2 L(T) + \ve(T) + \xi(T).
\ee
Moreover, equality holds if and only if there exists a linear order $OV(T): u_{0}, u_{1}, \ldots, u_{p-1}$ of the vertices of $T$ such that the following conditions are satisfied:
\begin{enumerate}[\rm (a)]
\item
\begin{enumerate}[\rm (i)]
\item in the case when $|W(T)| =  1$, the following hold: if $|S(T)|$ is odd, then $L(u_{0})+L(u_{p-1}) = 1$ and $OV(T)$ is admissible; if $|S(T)|$ is even, then either $L(u_{0})+L(u_{p-1}) = 1$ and $OV(T)$ is feasible or admissible, or $L(u_{0})+L(u_{p-1}) = 2$ and $OV(T)$ is admissible;
\item in the case when $|W(T)| = 2$, $OV(T)$ is admissible and the following hold: if $|S(T)| = 1$ or 2, then $L(u_{0})+L(u_{p-1}) = 0$; if $|S(T)| \geq 3$ is odd, then $L(u_{0})+L(u_{p-1}) = 1$; if $|S(T)| \ge 4$ is even, then $L(u_{0})+L(u_{p-1}) = 0$ or $2$;
\end{enumerate}
\item for $0 \leq i < j \leq p-1$, the distance $d(u_{i},u_{j})$ between $u_{i}$ and $u_{j}$ in $T$ satisfies
\be
\label{eqn:duv}
d(u_{i},u_{j}) \geq \sum_{t=i}^{j-1}\{L(u_{t})+L(u_{t+1})-a_t-(d+\ve(T))\} + (d+1),
\ee
where $a_0,a_1,\ldots,a_{p-2}$ are the non-negative integers determined by
\begin{equation}
\label{eq:a0}
a_0 = 0
\end{equation}
\begin{equation}
\label{eq:seq}
a_{t-1} + a_{t} =
\begin{cases}
|W(T)|, & \mbox{if $u_{t} \in S(T)$ and $\{u_{t-1}, u_{t+1}\} \cap W(T) = \emptyset$} \\
0, & \mbox{otherwise}
\end{cases}
\end{equation}
for $1 \le t \le p-2$.
\end{enumerate}
Furthermore, under conditions (a) and (b) the mapping $f: V(T) \rightarrow \{0, 1, 2, \ldots\}$ defined by
\be
\label{f00}
f(u_{0}) = 0
\ee
\be
\label{f11}
f(u_{i+1}) = f(u_{i}) - (L(u_{i}) + L(u_{i+1})) + a_{i} + (d + \ve(T)),\ 0 \leq i \leq p-2
\ee
is an optimal radio labelling of $T$.
\end{Theorem}

\begin{Remark}
\label{rem:seq}
{\em
Using the fact that $OV(T)$ satisfies the conditions in (a), we can prove that there is a unique sequence of integers $a_0,a_1,\ldots,a_{p-2}$ satisfying \eqref{eq:a0} and \eqref{eq:seq}, and moreover $a_i \in \{0, |W(T)|\}$ for $0 \le t \le p-2$. Of course, this sequence of non-negative integers is uniquely determined by the linear order $OV(T)$.
}
\end{Remark}

\begin{Remark}
\label{rem:pn}
{\em
In \cite{Liu}, Liu and Zhu proved that for the path $P_{n}$ of order $n \geq 4$ we have
\begin{equation}
\label{rn:path}
\rn(P_{n}) =
\begin{cases}
2k^{2}+2, & \mbox{if  } n = 2k+1 \\
2k(k-1)+1, & \mbox{if  } n = 2k.
\end{cases}
\end{equation}
This result can be derived from Theorem \ref{thm:lower}. In fact, $P_{n}$ is a two-branch tree with one weight center when $n$ is odd and two weight centers when $n$ is even. It can be verified that the linear order of the vertices of $P_n$ given in \cite[Theorem 3]{Liu} satisfies the necessary and sufficient condition in Theorem \ref{thm:lower}. Thus, $\rn(P_{n})$ is equal to the right-hand side of \eqref{rn:lower}, yielding \eqref{rn:path} exactly. In fact, $P_{n}$ has order $p = n$ and diameter $d = n-1$. Also, $\xi(P_{n}) = 2$, and $L(P_{n})$ is equal to $(n^{2}-1)/4$ when $n$ is odd and $n(n-2)/4$ when $n$ is even. Substituting all these into \eqref{rn:lower}, we obtain \eqref{rn:path} immediately.
}
\end{Remark}

The lower bound in \eqref{rn:lower} is a significant improvement of the one in \eqref{eqn:lb} as the difference between them (namely, $\xi(T)$) can be arbitrarily large. We illustrate this by constructing a family of caterpillars $T$ for which $\xi(T)$ can be arbitrarily large and the lower bound in \eqref{rn:lower} is attained. In general, a \emph{caterpillar} is a tree for which the removal of all its leaves results in a path called the spine. Given integers $n \ge 3$ and $k \ge 1$, define $C(n,k)$ to be the caterpillar with spine $v_1,v_2,\ldots,v_n$, where $v_i$ is adjacent to $v_{i+1}$ for $1 \leq i \leq n-1$, such that the following hold: (i) if $n$ is odd, then each of $v_1, v_{(n-1)/2}, v_{(n+3)/2}$ and $v_n$ has exactly $k$ neighbours of degree one and all other vertices on the spine have no neighbour of degree one; (ii) if $n$ is even, then each of $v_1, v_{(n-2)/2}, v_{(n+4)/2}$ and $v_n$ has exactly $k$ neighbours of degree one and all other vertices on the spine have no neighbour of degree one. Thus $C(n,k)$ has order $p = n + 2k \min\{2, \lfloor (n-1)/2 \rfloor\}$ and diameter $n+1$. Note that
$$
\ve(C(n,k)) =
\begin{cases}
1, & \mbox{ if $n$ is odd} \\
0, & \mbox{ if $n$ is even}.
\end{cases}
$$
and
\begin{equation*}
\xi(C(n,k)) =
\begin{cases}
k, & \mbox{ if $n$ is odd} \\
2(k-1), & \mbox{ if $n$ is even}.
\end{cases}
\end{equation*}
Since $k \geq 1$ is an arbitrary integer, $\xi(C(n,k))$ can be made arbitrarily large.

Our third result in this section is as follows.

\begin{Theorem}
\label{thm:cat}
Let $n \geq 3$ and $k \geq 1$ be integers. Then
\begin{equation}
\label{rn:cat}
\rn(C(n,k)) =
\begin{cases}
3k+7, & \mbox{if $n = 3$} \\
4k+11, & \mbox{if $n = 4$} \\
\frac{1}{2}(n^2+4nk+2n-2k-1), & \mbox{if $n > 3$ is odd} \\
\frac{1}{2}(n^2+4nk+2n-4k-6), & \mbox{if $n > 4$ is even}.
\end{cases}
\end{equation}
In particular, $\rn(C(n,k))$ attains the lower bound in \eqref{rn:lower}.
\end{Theorem}

We will prove Theorem  \ref{thm:nlb} in the next subsection. As preparations for the proof of Theorem \ref{thm:lower}, we will prove several lemmas in Subsection \ref{subsec:lms}. The proofs of Theorems \ref{thm:lower} and \ref{thm:cat} will be given in Subsections \ref{subsec:lb} and \ref{subsec:cat}, respectively.

\subsection{Proof of Theorem \ref{thm:nlb}}
\label{subsec:nlb}

\begin{proof} [Proof of Theorem \ref{thm:nlb}]
Suppose to the contrary that there exists a two-branch tree $T$ with order $p$ and diameter $d \geq 2$ such that $|S(T)| > |W(T)|$ but $\rn(T) = (p-1)(d+\ve(T))-2L(T)+\ve(T)$. Then there exists a linear order $u_{0},u_{1}, \ldots, u_{p-1}$ of the vertices of $T$ such that (a) and (b) in Theorem \ref{thm:ddb} are satisfied and moreover the mapping defined by \eqref{f0} and \eqref{f1} is an optimal radio labelling of $T$. Let $T_{1}$ and $T_{2}$ be the two branches of $T$. By Lemma \ref{lem3}, $|V(T_1)|-|V(T_2)| = \pm 1$ if $|W(T)|=1$ and $|V(T)|$ is even, and $|V(T_1)| = |V(T_2)|$ otherwise.

Assume first that $|W(T)| = 1$, say, $W(T) = \{w\}$. Then $|S(T) \cap V(T_1)| + |S(T) \cap V(T_2)| = |S(T)| > |W(T)| = 1$. This together with $u_{0} = w$ and $u_{p-1} \in N(w)$ (by condition (a) in Theorem \ref{thm:ddb}) implies that there exists a vertex $u_{t} \in V(T)$ such that $u_{t-1}, u_{t+1} \neq w$ and $L(u_{t}) \geq \lceil d/2 \rceil$. Let $L(u_{t-1}) = a$ and $L(u_{t+1}) = b$. By \eqref{eq:dij}, we have $d(u_{t-1},u_{t+1}) \ge L(u_{t-1})+2L(u_{t})+L(u_{t+1})-(d+1) \geq a+2(d/2)+b-(d+1) = a+b-1 > a+b-2\phi(u_{t-1}, u_{t+1}) = d(u_{t-1},u_{t+1})$ as $\phi(u_{t-1},u_{t+1}) \geq 1$, but this is a contradiction.

Now assume that $|W(T)| = 2$, say, $W(T) = \{w, w'\}$. Then $|S(T) \cap V(T_1)| + |S(T) \cap V(T_2)| = |S(T)| > |W(T)| = 2$. This together with $\{u_{0}, u_{p-1}\} = \{w,w'\}$ (by condition (a) in Theorem \ref{thm:ddb}) implies that there exists a vertex $u_{t} \in V(T)$ such that $u_{t-1}, u_{t+1} \not\in \{w, w'\}$ and $L(u_{t}) \geq \lfloor d/2 \rfloor$. Let $L(u_{t-1}) = a$ and $L(u_{t+1}) = b$. By \eqref{eq:dij}, we have $d(u_{t-1},u_{t+1}) \ge L(u_{t-1})+2L(u_{t})+L(u_{t+1})-d+1 = a+2((d-1)/2)+b-d+1 = a+b > a+b-2\phi(u_{t-1},u_{t+1}) = d(u_{t-1},u_{t+1})$ as $\phi(u_{t-1},u_{t+1}) \geq 1$, which is a contradiction.
\end{proof}

\subsection{Lemmas}
\label{subsec:lms}

\begin{Lemma}
\label{lem5}
Let $T$ be a two-branch tree with order $p$ and diameter $d \geq 2$. Let $f$ be a radio labelling of $T$ and $u_{0}, u_{1}, \ldots, u_{p-1}$ be the corresponding linear order as defined in \eqref{eqn:ord}. Let $i$ be an integer with $1 \leq i \leq p-3$ such that
\begin{enumerate}[\rm (a)]
  \item $u_{i}, u_{i+2} \not\in W(T)$, and $u_{i+1} \in S(T)$; and
  \item for $t=i,i+1$, $u_{t}$ and $u_{t+1}$ are in different branches when $|W(T)| = 1$ and opposite branches when $|W(T)| = 2$.
\end{enumerate}
Then
\begin{equation}
\label{eqn:jvp}
J_{f}(u_i,u_{i+1},u_{i+2}) \geq |W(T)|.
\end{equation}
\end{Lemma}

\begin{proof}
Since $f$ is a radio labelling, we have $f(u_{i+2}) - f(u_{i}) \geq d + 1 - d(u_{i},u_{i+2})$. Combining this and \eqref{eq:jij}, we obtain
\bea
J_{f}(u_i,u_{i+1},u_{i+2})
& = & f(u_{i+2}) - f(u_{i}) - 2(d+1) + d(u_{i},u_{i+1}) + d(u_{i+1},u_{i+2})  \non \\
& \geq & d(u_{i},u_{i+1}) + d(u_{i+1},u_{i+2}) - d(u_{i},u_{i+2}) - (d+1). \label{eqn:jvp2}
\eea
Since $u_{i+1} \in S(T)$, we have $L(u_{i+1}) \geq (d-(|W(T)|-1))/2$. Note that $u_{i}$ and $u_{i+2}$ are in the same branch as condition (b) is satisfied and $T$ is a two-branch tree. Note also that $u_i, u_{i+2} \not\in W(T)$ by (a). Hence $\phi(u_{i},u_{i+2}) \geq 1$ and $d(u_{i},u_{i+2}) = L(u_{i})+L(u_{i+2})-2\phi(u_{i},u_{i+2})$ by \eqref{eqn:dist}. Also, by condition (b), we have $d(u_{t},u_{t+1}) = L(u_{t})+L(u_{t+1}) + (|W(T)|-1)$ for $t = i,i+1$. Plugging all these into \eqref{eqn:jvp2}, we obtain
\bean
J_{f}(u_i,u_{i+1},u_{i+2}) & \geq & 2 L(u_{i+1}) + 2(|W(T)|-1) + 2\phi(u_{i},u_{i+2}) - (d+1) \\
& \geq & (d-(|W(T)|-1)) + 2(|W(T)|-1) + 2 - (d+1) \\
& = & |W(T)|
\eean
as required.
\end{proof}

\begin{Lemma}
\label{lem6}
Let $T$ be a two-branch tree with order $p$ and diameter $d \geq 2$. Let $f$ be a radio labelling of $T$ and $u_{0},u_{1},\ldots,u_{p-1}$ be the corresponding linear order as defined in \eqref{eqn:ord}. Let $i$ and $j$ be integers with $0 < i \leq j < p-1$ such that
\begin{enumerate}[\rm (a)]
\item $u_{t} \not\in W(T)$ for $i \leq t \leq j$, and $\{u_{i-1}, u_{i}, u_{j}, u_{j+1}\} \cap S(T) = \{u_{i}, u_{j}\}$; and
\item for $i-1 \leq t \leq j$, $u_{t}$ and $u_{t+1}$ are in different branches when $|W(T)| = 1$ and opposite branches when $|W(T)| = 2$.
\end{enumerate}
Then
\begin{equation}
\label{eqn:jvp3}
J_{f}(u_{i-1}, \ldots, u_{j+1}) \geq
\frac{|\{u_i, \ldots, u_j\} \cap S(T)| - |\{u_{i-1},u_{j+1}\} \cap W(T)|}{3-|W(T)|}.
\end{equation}
\end{Lemma}

\begin{proof}
We deal with the cases $|W(T)| = 1, 2$ separately.

\medskip
\textsf{Case 1:}~$|W(T)| = 1$.

\medskip
\textsf{Subcase 1.1:}~$|\{u_{i-1},u_{j+1}\} \cap W(T)| = 0$.

Consider first the case when $\{u_i, \ldots, u_j\} \subseteq S(T)$. In this case, \eqref{eqn:jvp3} becomes
\be
\label{eq:ij1}
J_{f}(u_{i-1}, \ldots, u_{j+1}) \geq \frac{j-i+1}{2}
\ee
and we prove this inequality by induction on $j-i+1$. In the case when $j-i+1 = 1$, since $u_{i} \in S(T)$, by \eqref{eqn:jvp} we have $J_{f}(u_{i-1},u_i,u_{i+1}) = 1 > 1/2$. In the case when $j-i+1 = 2$, since $u_i,u_{i+1} \in S(T)$, by \eqref{eqn:jvp} and the fact that $J_f(u_{i+1},u_{i+2}) \geq 0$, we have $J_f(u_{i-1},u_i,u_{i+1},u_{i+2}) \ge J_f(u_{i-1},u_i,u_{i+1}) \geq 1$. In the case when $j-i+1 = 3$, since $u_i,u_{i+1},u_{i+2} \in S(T)$,  by \eqref{eqn:jvp} we have $J_f(u_{i-1},\ldots,u_{i+3}) = J_f(u_{i-1},u_{i},u_{i+1})+J_f(u_{i+1},u_{i+2},u_{i+3}) = 2 \geq 3/2$. Hence \eqref{eq:ij1} holds when $j-i+1 = 1, 2$ or $3$. Now, for $j-i+1 \ge 4$, since $u_{t} \not\in W(T)$ for $i \leq t \leq j$ by our assumption, it follows from \eqref{eqn:jvp} that $J_{f}(u_{j-1}, u_{j}, u_{j+1}) \geq 1$. Hence $J_{f}(u_{i-1}, \ldots, u_{j+1}) = J_{f}(u_{i-1}, \ldots, u_{j-1}) + J_{f}(u_{j-1}, u_{j}, u_{j+1}) \geq J_{f}(u_{i}, \ldots, u_{j-1}) + 1$. Using this and the fact that \eqref{eq:ij1} holds when $j-i+1 = 1, 2$ or $3$, we can easily prove \eqref{eq:ij1} by induction on $j-i+1$.

Now consider the case when $\{u_i, \ldots, u_j\} \not\subseteq S(T)$. Let $U_{1}, \ldots, U_{l}$ be the maximal segments of consecutive vertices in the ordered set $\{u_i, \ldots, u_j\} \cap S(T)$. Let $U'_t$ be the segment obtained from $U_t$ by adding the vertex immediately before the first vertex of $U_t$ and the vertex immediately after the last vertex of $U_t$, for $1 \le t \le l$. Applying \eqref{eq:ij1} to $U'_t$, we obtain $J_{f}(U'_{t}) \geq |U_{t}|/2$ for $1 \le t \le l$. Thus $J_{f}(u_{i-1}, \ldots, u_{j+1}) \geq \sum_{t=1}^{l}J_{f}(U'_{t}) \geq \sum_{t=1}^{l} |U_{t}|/2 = |\{u_{i}, \ldots, u_{j}\} \cap S(T)|/2$ as claimed in \eqref{eqn:jvp3}.

\medskip
\textsf{Subcase 1.2:}~$|\{u_{i-1},u_{j+1}\} \cap W(T)| = 1$.

Consider first the case when $W(T) = \{u_{i-1}\}$. If $i = j$, then $|\{u_i, \ldots, u_j\}| = 1$ and \eqref{eqn:jvp3} becomes $J_{f}(u_{i-1}, \ldots, u_{j+1}) \geq 0$, which is trivially true. Assume $i < j$. Let $i^{*}$ be the smallest subscript such that $i^{*} > i$ and $u_{i^{*}} \in S(T)$. Since $i < j$ and $u_{j} \in S(T)$, $i^{*}$ is well defined. Since $u_{i^{*}} \not\in W(T)$, we may apply the result obtained in Subcase 1.1 to the segment $\{u_{i^{*}-1},\ldots,u_{j+1}\}$, yielding $J_{f}(u_{i-1}, \ldots, u_{j+1}) \geq J_{f}(u_{i^{*}-1}, \ldots, u_{j+1}) \geq |\{u_{i^{*}}, \ldots, u_j\} \cap S(T)|/2 = (|\{u_i, \ldots, u_j\} \cap S(T)|-1)/2$, as desired. The case when $W(T) = \{u_{j+1}\}$ can be treated similarly.

\medskip
\textsf{Case 2:}~$|W(T)| = 2$.

\medskip
\textsf{Subcase 2.1:}~$|\{u_{i-1},u_{j+1}\} \cap W(T)| = 0$.

Consider first the case when $\{u_{i},\ldots,u_{j}\} \subseteq S(T)$. In this case, we prove
\be
\label{eq:ij2}
J_{f}(u_{i-1},\ldots,u_{j+1}) \geq j-i+1
\ee
by induction on $j-i+1$.

In fact, if $j-i+1 = 1$, then Lemma \ref{lem5} applies to $(u_{i-1},u_{i},u_{i+1})$ as $u_{i} \in S(T)$. So we have $J_{f}(u_{i-1},u_{i},u_{i+1}) \geq 2$ by \eqref{eqn:jvp}. Similarly, if $j-i+1 = 2$, then we have $J_{f}(u_{i-1},\ldots,u_{j+1}) = J_{f}(u_{i-1}, u_{i}, u_{i+1}) + J_f(u_{i+1}, u_{i+2}) \ge J_{f}(u_{i-1}, u_{i}, u_{i+1}) \geq 2$ by \eqref{eqn:jvp} and $J_f(u_{i+1},u_{i+2}) \geq 0$. Hence \eqref{eq:ij2} is true when $j-i+1 = 1$ or $2$.  In general, for $j-i+1 \ge 3$, Lemma \ref{lem5} applies to $(u_{j-1}, u_{j}, u_{j+1})$ as $u_{t} \not\in W(T)$ for $i \leq t \leq j$, and hence $J_{f}(u_{j-1}, u_{j}, u_{j+1}) \ge 2$ by \eqref{eqn:jvp}. So  $J_{f}(u_{i-1},\ldots,u_{j+1}) = J_{f}(u_{i-1},\ldots,u_{j-1}) + J_{f}(u_{j-1}, u_{j}, u_{j+1}) \ge J_{f}(u_{i-1},\ldots,u_{j-1}) + 2$. Using this inequality and the fact that \eqref{eq:ij2} is true when $j-i+1 = 1$ or $2$, we can easily obtain \eqref{eq:ij2} by induction on $j-i+1$.

Now consider the case when $\{u_{i},\ldots,u_{j}\} \not\subseteq S(T)$. Let $U_{1},\ldots,U_{l}$ be the maximal segments of consecutive vertices in the ordered set $\{u_{i},\ldots,u_{j}\} \cap S(T)$. Let $U'_t$ be the segment obtained from $U_t$ by adding the vertex immediately before the first vertex of $U_t$ and the vertex immediately after the last vertex of $U_t$, for $1 \le t \le l$. Applying \eqref{eq:ij2} to $U'_t$, we obtain $J_{f}(U'_{t}) \geq |U_t|$ for $1 \le t \le l$. Hence $J_{f}(u_{i-1},\ldots,u_{j+1}) \geq \sum_{t=1}^{l}J_{f}(U'_{t}) \geq \sum_{t=1}^{l}|U_{t}| = |\{u_{i},\ldots,u_{j}\} \cap S(T)|$ as desired.

\medskip
\textsf{Subcase 2.2:}~$|\{u_{i-1},u_{j+1}\} \cap W(T)| = 1$.

Without loss of generality we may assume $\{u_{i-1},u_{j+1}\} \cap W(T) = \{u_{i-1}\}$ since the case when $\{u_{i-1},u_{j+1}\} \cap W(T) = \{u_{j+1}\}$ can be treated similarly. If $i = j$, then \eqref{eqn:jvp3} becomes $J_{f}(u_{i-1}, u_{i}, u_{i+1}) \geq 0$, which is trivially true. Assume then $i < j$. Let $i^{*}$ be the smallest subscript with $i < i^{*}$ such that $u_{i^{*}} \in S(T)$. Since $i < j$ and $u_{j} \in S(T)$, $i^{*}$ is well defined, and moreover $u_{i^{*}-1} \not\in W(T)$ as $u_t \not\in W(T)$ for $i \leq t \leq j$. So we can apply the result obtained in Subcase 2.1 to the segment $U' = \{u_{i^{*}},\ldots, u_{j}\}$, yielding $J_{f}(u_{i-1}, \ldots, u_{j+1}) \geq J_{f}(u_{i^{*}-1}, \ldots, u_{j+1}) \geq |U' \cap S(T)| = |\{u_{i^{*}}, \ldots, u_{j}\} \cap S(T)| = |\{u_i, \ldots, u_j\} \cap S(T)|-1$, as desired.

\medskip
\textsf{Subcase 2.3:}~$|\{u_{i-1}, u_{j+1}\} \cap W(T)| = 2$.

In this case we have $W(T) = \{u_{i-1}, u_{j+1}\}$.
If $j - i + 1 = 1$ or $2$, then \eqref{eqn:jvp3} is trivially true. Assume $j-i+1 \geq 3$. Let $i^{*}$ be the smallest subscript such that $i^{*} > i$ and $u_{i^{*}} \in S(T)$, and let $j^{*}$ be the largest subscript such that $j^{*} < j$ and $u_{j^{*}} \in S(T)$. Then $i^{*} < j^{*}$, $u_{i^{*}-1} \not\in W(T)$ and $u_{j^{*}+1} \not \in W(T)$ as $u_t \not\in W(T)$ for $i \leq t \leq j$. So we can apply the result obtained in Subcase 2.1 to the segment $U'= \{u_{i^{*}}, \ldots, u_{j^{*}}\}$, yielding $J_{f}(u_{i-1}, \ldots, u_{j+1}) \geq J_{f}(u_{i^{*}-1}, \ldots, u_{j^{*}+1}) \geq |U' \cap S(T)| = |\{u_{i}, \ldots, u_{j}\} \cap S(T)|-2$.
\end{proof}

\begin{Lemma}
\label{lem7}
Let $T$ be a two-branch tree of order $p$ and diameter $d \geq 2$. Let $f$ be a radio labelling of $T$ and $u_{0},u_{1},\ldots,u_{p-1}$ the corresponding linear order as defined in \eqref{eqn:ord}. Let $i$ and $j$ be integers with $0 < i \leq j < p-1$ such that
\begin{enumerate}[\rm (a)]
  \item $W(T) \cap \{u_{i+1},\ldots,u_{j-1}\} \neq \emptyset$; and
  \item for $i-1 \leq t \leq j$, $u_{t}$ and $u_{t+1}$ are in different branches when $|W(T)| = 1$ and opposite branches when $|W(T)| = 2$.
\end{enumerate}
Then
\begin{equation*}
\label{eqn:jvp5}
J_{f}(u_{i}, \ldots, u_{j}) \geq \frac{|\{u_{i}, \ldots, u_{j}\} \cap S(T)| - 2|\{u_{i+1},\ldots,u_{j-1}\} \cap W(T)|}{3-|W(T)|}.
\end{equation*}
\end{Lemma}

\begin{proof}
Assume first that $|W(T)| = 1$. Then by (a) we have $W(T) = \{u_{i^{*}}\}$ for some $i^{*}$ with $i < i^{*} < j$. The segments $U_{1} = \{u_{i}, \ldots, u_{i^{*}-1}\}$ and $U_{2} = \{u_{i^{*}+1}, \ldots, u_{j}\}$ both satisfy conditions (a) and (b) in Lemma \ref{lem6}. Thus, by Lemma \ref{lem6}, we have $J_{f}(u_{i}, \ldots, u_{j}) \geq J_{f}(U_{1})+J_{f}(U_{2}) \geq (|U_{1} \cap S(T)|-1)/2+(|U_{2} \cap S(T)|-1)/2$ = $(|\{u_{i}, \ldots, u_{j}\} \cap S(T)|-2)/2$ as desired.

Now assume that $|W(T)| = 2$. By condition (a), if $|\{u_{i+1},\ldots,u_{j-1}\} \cap W(T)| = 1$, then we have $\{u_{i+1},\ldots,u_{j-1}\} \cap W(T) = \{u_{i^*}\}$ for some $i < i^* < j$. The segment $U_1 = \{u_i,\ldots,u_{i^*-1}\}$ and $U_2 = \{u_{i^*+1},\ldots,u_{j-1}\}$ satisfy conditions (a) and (b) in Lemma \ref{lem6}. Thus, by Lemma \ref{lem6}, we have $J_f(u_i,\ldots,u_j) \geq J_f(U_1)+J_f(U_2) \geq (|U_1 \cap S(T)|-1)+(|U_2 \cap S(T)|-1) = |\{u_i,\ldots,u_j\} \cap S(T)|-2$. By (a) if $|\{u_{i+1},\ldots,u_{j-1}\} \cap W(T)| = 2$ then we have $\{u_{i+1},\ldots,u_{j-1}\} \cap W(T) = \{u_{i^{*}},u_{j^{*}}\}$ for some $i^{*}, j^{*}$ with $i < i^{*} < j^{*} < j$. The segments $U_{1} = \{u_{i},\ldots,u_{i^{*}-1}\}$, $U_{2} = \{u_{i^{*}+1},\ldots,u_{j^{*}-1}\}$ and $U_{3} = \{u_{j^{*}+1},\ldots,u_{j}\}$ all satisfy conditions (a) and (b) in Lemma \ref{lem6}. Thus, by Lemma \ref{lem6}, we have $J_{f}(u_{i}, \ldots, u_{j}) \geq J_{f}(U_{1}) + J_{f}(U_{2}) + J_{f}(U_{3}) \geq (|U_{1} \cap S(T)|-1) + (|U_{2} \cap S(T)|-2) + (|U_{3} \cap S(T)|-1) = |\{u_{i}, \ldots, u_{j}\} \cap S(T)|-4$ as desired.
\end{proof}

\begin{Lemma}
\label{lem8}
Let $T$ be a two-branch tree of order $p$ and diameter $d \geq 2$. Let $f$ be a radio labelling of $T$ and $OV_{f}(T): u_{0},u_{1},\ldots,u_{p-1}$ the corresponding linear order as defined in \eqref{eqn:ord}. Suppose that, for $0 \leq i \leq p-2$, $u_{i}$ and $u_{i+1}$ are in different branches when $|W(T)| = 1$ and opposite branches when $|W(T)| = 2$. Then
\begin{equation}
\label{eq:jvt}
J_{f}(T) \geq
\begin{cases}
\big\lfloor \frac{|S(T)|}{2} \big\rfloor, & \mbox{ if } |W(T)| = |\{u_{0}, u_{p-1}\} \cap W(T)| = 1 \\
\big\lceil \frac{|S(T)|}{2}\big\rceil-1, & \mbox{ if } |W(T)| = 1 \mbox{ but } |\{u_{0}, u_{p-1}\} \cap W(T)| = 0 \\
|S(T)| + |\{u_{0}, u_{p-1}\} \cap W(T)| - 4, & \mbox{ if } |W(T)| = 2.
\end{cases}
\end{equation}
Moreover, equality in \eqref{eq:jvt} holds if and only if for $1 \le i \le p-2$,
\begin{equation}\label{eq:Jfu}
J_{f}(u_{i-1},u_{i},u_{i+1}) =
\begin{cases}
|W(T)|, & \mbox{ if } u_{i} \in S(T) \mbox{ and } \{u_{i-1}, u_{i+1}\} \cap W(T) = \emptyset \\
0, & \mbox{ otherwise},
\end{cases}
\end{equation}
and $OV_{f}(T)$ is admissible in each case, with the exception of the first line in \eqref{eq:jvt} for which $OV_{f}(T)$ may be feasible or admissible when $|S(T)|$ is even.
\end{Lemma}

\begin{proof}
We consider the following two cases separately.

\medskip
\textsf{Case 1:}~$|W(T)|=1$.

\medskip
\textsf{Subcase 1.1:}~$|\{u_0,u_{p-1}\} \cap W(T)| = 1$.

Without loss of generality we may assume $u_0 \in W(T)$. Setting $i = 1$ and $j = p-1$ in Lemma \ref{lem6}, we obtain $J_f(T) \geq (|S(T)|-1)/2$, which implies $J_f(T) \geq \lfloor |S(T)|/2 \rfloor$ as $J_f(T)$ is an integer.

\medskip
\textsf{Subcase 1.2:}~$|\{u_0,u_{p-1}\} \cap W(T)| = 0$.

Setting $i = 0$ and $j = p-1$ in Lemma \ref{lem7}, we obtain $J_f(T) \geq (|S(T)|/2)-1$, which implies $J_f(T) \geq \lceil |S(T)|/2 \rceil - 1$ as $J_f(T)$ is an integer.

\medskip
\textsf{Case 2:}~$|W(T)| = 2$.

\medskip
\textsf{Subcase 2.1:}~$|\{u_0,u_{p-1}\} \cap W(T)| = 2$.

Setting $i = 1$ and $j = p-2$ in Lemma \ref{lem6}, we obtain $J_f(T) \geq |S(T)|-2$ as required.

\medskip
\textsf{Subcase 2.2:}~$|\{u_0,u_{p-1}\} \cap W(T)| = 1$.

Without loss of generality we may assume $u_0 \in W(T)$. Then $u_{i^*} \in W(T)$ for some $0 < i^* \leq p-1$ as $|W(T)|=2$. If $i^* = 1$, then setting $i=2$ and $j = p-1$ in Lemma \ref{lem6}, and noting $S(T) \subseteq \{u_2,\ldots,u_{p-1}\}$, we obtain $J_f(T) \geq J_f(U) \geq |S(T)|-1 > |S(T)|-3$. If $i^* = 2$, then let $U = \{u_1,\ldots,u_{p-1}\}$. Since $S(T) \subseteq U$, by Lemma \ref{lem7} we obtain $J_f(T) \geq J_f(U) \geq |U \cap S(T)|-2 > |S(T)|-3$. If $i^* \geq 3$, then let $U_1 = \{u_0,u_1,u_2\}$ and $U_2 = \{u_2,\ldots,u_{p-1}\}$. Note that $|U_1 \cap S(T)|+|U_2 \cap S(T)| \ge |S(T)|$ as $U_1 \cup U_2 = V(T)$. Applying Lemmas \ref{lem6} and \ref{lem7} to $U_1$ and $U_2$, respectively, we obtain $J_f(T) = J_f(U_1) + J_f(U_2) \geq (|U_1 \cap S(T)|-1)+(|U_2 \cap S(T)|-2) \ge |S(T)|-3$.

\medskip
\textsf{Subcase 2.3:}~$|\{u_0,u_{p-1}\} \cap W(T)| = 0$.

Assume that $u_{i^*}, u_{j^*} \in W(T)$ with $0 < i^* < j^* < p-1$. If $j^*=i^*+1$, then let $U_1 = \{u_0,\ldots,u_{i^*-1}\}$ and $U_2 = \{u_{j^*+1},\ldots,u_{p-1}\}$. Note that $(U_1 \cap S(T)) \cup (U_2 \cap S(T)) = S(T)$. By Lemma \ref{lem6}, we obtain $J_f(T) \geq  J_f(U_1)+J_f(U_2) \geq (|U_1 \cap S(T)|-1)+(|U_2 \cap S(T)|-1) \ge |S(T)|-2 > |S(T)|-4$. If $j^* = i^*+2$, then let $U_1 = \{u_0,\ldots,u_{i^*-1}\}$ and $U_2 = \{u_{i^*+1},\ldots,u_{p-1}\}$. We have $(U_1 \cap S(T)) \cup (U_2 \cap S(T)) = S(T)$. Applying Lemmas \ref{lem6} and \ref{lem7} to $U_1$ and $U_2$, respectively, we obtain $J_f(T) \geq J_f(U_1)+J_f(U_2) \geq (|U_1 \cap S(T)|-1)+(|U_2 \cap S(T)|-2) \ge |S(T)|-3 > |S(T)|-4$. If $j \geq i^*+3$, then let $U_1 = \{u_0,\ldots,u_{i^*+1}\}$ and $U_2 = \{u_{i^*+2},\ldots,u_{p-1}\}$. Again we have $(U_1 \cap S(T)) \cup (U_2 \cap S(T)) = S(T)$. By Lemma \ref{lem7}, we obtain $J_f(T) \geq  J_f(U_1)+J_f(U_2) \geq (|U_1 \cap S(T)|-2)+(|U_2 \cap S(T)|-2) \ge |S(T)|-4$.

So far we have completed the proof of \eqref{eq:jvt}. Since $J_f(T)=\sum_{i=0}^{p-2}J_f(u_i,u_{i+1})$ and $J_f(u_i,u_{i+1})+J_f(u_{i+1},u_{i+2}) = J_f(u_i,u_{i+1},u_{i+2}) \ge |W(T)|$ by \eqref{eqn:jvp}, one can see that equality in \eqref{eq:jvt} holds if and only if \eqref{eq:Jfu} is satisfied and the statements below \eqref{eq:Jfu} hold.
\end{proof}

Consider a radio labelling $f$ of $T$ and the corresponding linear order $OV_{f}(T): u_{0}, u_{1}, \ldots, u_{p-1}$. A pair of consecutive vertices $u_{i}, u_{i+1}$ in the linear order $OV_{f}(T)$ is called \emph{bad} if $u_{i} \not\in W(T)$, $u_{i+1} \not\in W(T)$, and $u_{i}$ and $u_{i+1}$ are in the same branch of $T$. Denote by $X_{f}$ the set of vertices of $T$ which are in at least one bad pair in the linear order $OV_{f}(T)$. Denote by $\mt{F}$ the set of radio labellings $f$ of $T$ such that for each $0 \leq i \leq p-2$, the vertices $u_{i}, u_{i+1}$ in the linear order $OV_{f}(T)$ are in different branches when $|W(T)| = 1$ and in opposite branches when $|W(T)| = 2$. In other words, a radio labelling $f$ of $T$ is in $\mt{F}$ if and only if $\phi(u_i,u_{i+1})=0$ for $0 \leq i \leq p-2$ and also $\delta(u_i,u_{i+1})=1$ for $0 \leq i \leq p-2$ when $|W(T)|=2$. Alternatively, $f$ is in $\mt{F}$ if and only if there is no bad pair of vertices with respect to the linear order $OV_{f}(T)$.

\begin{Lemma}
\label{lem9}
Let $T$ be a two-branch tree with order $p$ and diameter $d \geq 2$. Let $f$ be a radio labelling of $T$ with $f \in \mt{F}$ and $u_{0},u_{1},\ldots,u_{p-1}$ the corresponding linear order of the vertices of $T$. Then
\begin{equation}
\label{eq:svp}
\sigma(f) \geq
\begin{cases}
\big\lfloor \frac{|S(T)|}{2} \big\rfloor, & \mbox{ if } |W(T)| = |\{u_{0}, u_{p-1}\} \cap W(T)| = 1 \\
\big\lceil \frac{|S(T)|}{2} \big\rceil - 1, & \mbox{ if } |W(T)| = 1 \mbox{ but } |\{u_{0}, u_{p-1}\} \cap W(T)| = 0 \\
|S(T)| + |\{u_{0}, u_{p-1}\} \cap W(T)| & \\
\quad -(p+3)\}, & \mbox{ if } |W(T)| = 2.
\end{cases}
\end{equation}
Moreover, equality holds if and only if the corresponding equality in \eqref{eq:jvt} holds.
\end{Lemma}

\begin{proof}
In view of \eqref{eq:delta}, Lemma \ref{lem4}(b) and the assumption that $f \in \mt{F}$, for $0 \leq i \leq p-2$, we have $\phi(u_{i}, u_{i+1}) = 0$, and $\d(u_{i},u_{i+1}) = 0$ or $1$ depending on whether $|W(T)| = 1$ or $2$. Thus, $\sigma(f) = J_{f}(T)$ if $|W(T)| = 1$, and $\sigma(f) = J_{f}(T)-(p-1)$ if $|W(T)| = 2$. The result then follows from Lemma \ref{lem8} immediately.
\end{proof}

\begin{Lemma}
\label{lem10}
Let $T$ be a two-branch tree with order $p$ and diameter $d \geq 2$. Let $f$ be a radio labelling of $T$ with $f \not\in \mt{F}$ and $u_{0},u_{1},\ldots,u_{p-1}$ the corresponding linear order of the vertices of $T$. Then
\begin{equation}
\label{eq:svpn}
\sigma(f) \geq
\begin{cases}
|X_f| + \big\lfloor \frac{|S(T) \setminus X_{f}|}{2} \big\rfloor, & \mbox{if } |W(T)| = |\{u_{0}, u_{p-1}\} \cap W(T)| = 1\\
|X_f| + \big\lceil \frac{|S(T) \setminus X_{f}|}{2} \big\rceil - 1, & \mbox{if } |W(T)| = 1 \mbox{ but } |\{u_{0}, u_{p-1}\} \cap W(T)| = 0 \\
2|X_f| + |S(T) \setminus X_{f}| & \\
\quad + |\{u_{0}, u_{p-1}\} \cap W(T)| - (p+3), & \mbox{if } |W(T)| = 2. \\
\end{cases}
\end{equation}
\end{Lemma}	

\begin{proof}
Since $f \not\in \mt{F}$, we have $X_f \neq \emptyset$ and $X_f \cap W(T) = \emptyset$. Assume $X_f = \{u_{i_1},u_{i_2},\ldots,u_{i_k}\}$, where $0 \leq i_1 < i_2 < \cdots < i_k \leq p-1$. Let $U_{i_t} = \{u_{i_{t}+1},\ldots,u_{i_{t+1}-1}\}$ for $1 \leq t \leq k-1$. Set $i_0 = 0$ and $i_{k+1} = p-1$. Let $U_{i_0} = \{u_1,\ldots,u_{i_1-1}\}$ if $u_0 \in W(T)$ and $U_{i_0} = \{u_0,\ldots,u_{i_1-1}\}$ otherwise, and let $U_{i_k} = \{u_{i_k+1},\ldots,u_{i_{k+1}-1}\}$ if $u_{p-1} \in W(T)$ and $U_{i_k} = \{u_{i_k+1},\ldots,u_{i_{k+1}}\}$ otherwise. Set $U = \cup_{t=0}^{k}U_{i_t}$. Then
\begin{equation}
\label{eq:jfx}
J_f(U) \geq \sum_{t=0}^{k} J_f(U_{i_t}).
\end{equation}
It is clear that $U \cap S(T) = S(T) \setminus X_f$. Since $f \not\in \mt{F}$, by the definitions of $\phi$ and $\delta$, the following hold for $0 \leq i \leq p-2$: if $|\{u_{i}, u_{i+1}\} \cap X_{f}| = 2$, then $\phi(u_{i},u_{i+1}) \geq 1$; if $|\{u_{i}, u_{i+1}\} \cap X_{f}| = 0$ or 1, then $\phi(u_{i},u_{i+1}) = 0$; if $\{u_{i}, u_{i+1}\} \cap S(T) = \emptyset$ or $|\{u_i,u_{i+1}\} \cap X_f| = 1$, then $J_{f}(u_{i},u_{i+1}) \geq 0$; if $|W(T)| = 1$, then $\d(u_{i},u_{i+1}) = 0$; if $|W(T)| = 2$, then $\d(u_{i},u_{i+1}) = 0$ or $1$ depending on whether $|\{u_{i}, u_{i+1}\} \cap X_{f}|=2$ or $|\{u_{i}, u_{i+1}\} \cap X_{f}| = 0$ or $1$. Hence by \eqref{eq:sigf} we have
\begin{equation}\label{eq:sixj}
  \sigma(f) \geq
\begin{cases}
  |X_f| + J_f(U), & \mbox{if $|W(T)|=1$} \\
  2|X_f|+J_f(U)-(p-1), & \mbox{if $|W(T)|=2$}.
\end{cases}
\end{equation}

\medskip
\textsf{Case 1:}~$|W(T)|=1$.

In the case when $|\{u_0,u_{p-1}\} \cap W(T)|=1$, without loss of generality may assume that $u_0 \in W(T)$. Applying Lemma \ref{lem6} to each $U_{i_t}$ for $0 \le t \le k$, we obtain $J_f(U_{i_0}) \geq (|U_{i_0} \cap S(T)|-1)/2$ and $J_f(U_{i_t}) \geq |U_{i_t} \cap S(T)|/2$ for $1 \le t \le k$.

In the case when $|\{u_0,u_{p-1}\} \cap W(T)|=0$, there exists $0 \leq s \leq k$ such that $U_{i_s} \cap W(T) \neq \emptyset$. Applying Lemma \ref{lem6} to each $U_{i_t}$ for $0 \leq t \leq k$ with $t \neq s$ and Lemma \ref{lem7} to $U_{i_s}$, we obtain $J_f(U_{i_t}) \geq |U_{i_t} \cap S(T)|/2$ for $0 \leq t \leq k$ with $t \neq s$ and $J_f(U_{i_s}) \geq (|U_{i_s} \cap S(T)|/2)-1$.

In either case above, we obtain the first line in \eqref{eq:svpn} by plugging the acquired inequalities above into \eqref{eq:jfx} and the first line in \eqref{eq:sixj}.

\medskip
\textsf{Case 2:}~$|W(T)| = 2$.

In the case when $|\{u_0,u_{p-1}\} \cap W(T)|=2$, by applying Lemma \ref{lem6} to each $U_{i_t}$ we obtain $J_f(U_{i_t}) \geq |U_{i_t} \cap S(T)|-1$ for $t=0,k$ and $J_f(U_{i_t}) \geq |U_{i_t} \cap S(T)|$ for $1 \leq t \leq k-1$.

In the case when $|\{u_0,u_{p-1}\} \cap W(T)|=1$, without loss of generality we may assume that $u_0 \in W(T)$ and $U_{i_s} \cap W(T) \neq \emptyset$ for some $0 \leq s \leq k$. Applying Lemma \ref{lem6} to each $U_{i_t}$ with $t \neq s$ and Lemma \ref{lem7} to $U_{i_s}$, we obtain $J_f(U_{i_0}) \geq |U_{i_0} \cap S(T)|-1$, $J_f(U_{i_t}) \geq |U_{i_t} \cap S(T)|$ for $1 \leq t \leq k$ with $t \neq s$, and $J_f(U_{i_s}) \geq |U_{i_s} \cap S(T)|-2$.

Finally, in the case when $|\{u_0,u_{p-1}\} \cap W(T)|=0$, we may assume that $U_{i_q} \cap S(T) \neq \emptyset$ and $U_{i_s} \cap S(T) \neq \emptyset$ for some $0 \leq q,s \leq k$. Applying Lemma \ref{lem6} to each $U_{i_t}$ with $t \neq q, s$ and Lemma \ref{lem7} to $U_{i_q}$ and $U_{i_s}$, we obtain $J_f(U_{i_t}) \geq |U_{i_t} \cap S(T)|$ for $0 \leq t \leq k$ with $t \neq q,s$ and $J_f(U_{i_t}) \geq |U_{i_t} \cap S(T)|-2$ for $t=q,s$.

In each case above, we obtain the third line in \eqref{eq:svpn} by plugging the acquired inequalities above into \eqref{eq:jfx} and the second line in \eqref{eq:sixj}.
\end{proof}

\begin{Remark}
\label{rem2}
{\em
Let $T$ be a two-branch tree with order $p \geq 5$ and branches $T_{1}$ and $T_{2}$. Let $f$ be a radio labelling of $T$ such that $f \not\in \mt{F}$. Note that if $X_{f}$ contains a pair of vertices $u_{i}, u_{i+1}$ of $T_{1}$, then it must contain some pair of vertices $u_{i^{*}}, u_{i^{*}+1}$ of $T_{2}$ as $T$ is a two-branch tree and,  by Lemma \ref{lem3}, $|V(T_1)|-|V(T_2)| = \pm 1$ when $|W(T)|=1$ and $|V(T_1)| = |V(T_2)|$ when $|W(T)| = 2$. Thus, for any $f \not\in \mt{F}$, $|X_{f}| \ge 4$ and $|X_{f}|$ is even. Consequently, the right-hand side of \eqref{eq:svpn} is larger than the right-hand side of \eqref{eq:svp} by at least $2$ in the first two lines and at least $4$ in the third line.
}
\end{Remark}

\subsection{Proof of Theorem \ref{thm:lower}}
\label{subsec:lb}

\begin{proof} [Proof of Theorem \ref{thm:lower}]
Let $T$ be a two-branch tree of order $p$ and diameter $d \geq 2$. Set $\ve = \ve(T)$ and $\xi = \xi(T)$.

We first prove the lower bound \eqref{rn:lower}. To achieve this it suffices to prove that any radio labelling of $T$ has span no less than the right-hand side of \eqref{rn:lower}. Suppose that $f$ is an arbitrary radio labelling of $T$ and let $OV_{f}(T): u_{0},u_{1},\ldots,u_{p-1}$ be the linear order of vertices of $T$ induced by $f$. Set
$$
\alpha(f) = L(u_{0}) + L(u_{p-1}) + \sigma(f),
$$
where $\sigma(f)$ is as defined in \eqref{eq:sigf}. Then
$$
\span(f) = (p-1)(d+1) - 2 L(T) + \alpha(f)
$$
by \eqref{eqn:sumup}.
So \eqref{rn:lower} is true if and only if $(p-1)(d+1) - 2 L(T) + \alpha(f) \ge (p-1)(d+\ve) - 2 L(T) + \ve + \xi$, or equivalently, $\alpha(f) \ge (p-1)(\ve - 1) + \ve + \xi$. In other words, to prove \eqref{rn:lower} it suffices to prove
\begin{equation}
\label{eq:alpha}
\alpha(f) \ge
\begin{cases}
\lfloor |S(T)|/2 \rfloor + 1, & \text{ if } |W(T)| = 1 \\
\max\{0, |S(T)|-2\} - (p-1), & \text{ if } |W(T)| = 2.
\end{cases}
\end{equation}
Moreover, equality in \eqref{rn:lower} holds with $f$ an optimal radio labelling of $T$ if and only if equality in \eqref{eq:alpha} holds.

Note that $L(u_{0})+L(u_{p-1}) \geq 1$ when $|W(T)| = 1$ and $L(u_{0})+L(u_{p-1}) \geq 0$ when $|W(T)| = 2$. Note also that, by Lemmas \ref{lem9} and \ref{lem10} and Remark \ref{rem2}, the minimum value of $\alpha(f)$ over all radio labellings $f$ is achieved by some labelling in $\mt{F}$. So it suffices to prove \eqref{eq:alpha} for any $f \in \mt{F}$. Henceforth we assume $f \in \mt{F}$.

\medskip
\textsf{Case 1:} $|W(T)| = 1$.

In this case, at most one of $u_{0}$ and $u_{p-1}$ can be in $W(T)$, and hence $L(u_{0})+L(u_{p-1}) \geq 1$.

\medskip
\textsf{Subcase 1.1:} $|S(T)|$ = 1.

By Lemma \ref{lem9}, we have $\alpha(f) \geq 1$ and so the first line in \eqref{eq:alpha} holds. Moreover, $\alpha(f) = 1$ holds if and only if $L(u_{0})+L(u_{p-1}) = 1$  (that is, one of $u_{0}$ and $u_{p-1}$ is in $W(T)$ and the other is in $N(W(T))$) and $\sigma(f) = 0$ (that is, $J_{f}(T) = 0$ and $\phi(u_t, u_{t+1}) = 0$ for $0 \le t \le p-2$). Thus, by the last statements in Lemmas \ref{lem8} and \ref{lem9}, $\alpha(f) = 1$ holds if and only if $L(u_0)+L(u_{p-1})=1$ and $OV_{f}(T)$ is admissible and satisfies \eqref{eq:Jfu}.

\medskip
\textsf{Subcase 1.2:} $|S(T)| \ge 2$.

Assume first $|\{u_{0}, u_{p-1}\} \cap W(T)| = 1$. Then by Lemma \ref{lem9} the first line in \eqref{eq:alpha} holds, and equality there occurs if and only if $L(u_{0})+L(u_{p-1}) = 1$ (that is, one of $u_{0}$ and $u_{p-1}$ is in $W(T)$ and the other is in $N(W(T))$) and $\sigma(f) = \lfloor |S(T)|/2 \rfloor$ (that is, $J_{f}(T) = \lfloor |S(T)|/2 \rfloor$ and $\phi(u_t, u_{t+1}) = 0$ for $0 \le t \le p-2$).

Now assume $|\{u_{0}, u_{p-1}\} \cap W(T)| = 0$. Then $L(u_{0})+L(u_{p-1}) \geq 2$ and $\alpha(f) \geq 2 + (|S(T)|/2)-1 = (|S(T)|/2)+1$ by Lemma \ref{lem9}. So the first line in \eqref{eq:alpha} holds. Moreover, equality in the first line in \eqref{eq:alpha} occurs if and only if $|S(T)|$ is even, $L(u_{0})+L(u_{p-1}) = 2$ (that is, both $u_{0}$ and $u_{p-1}$ are in $N(W(T))$), and $\sigma(f) = (|S(T)|/2) - 1$ (that is, $J_{f}(T) = (|S(T)|/2) - 1$ and $\phi(u_t, u_{t+1}) = 0$ for $0 \le t \le p-2$).

In summary, by the last statements in Lemmas \ref{lem8} and \ref{lem9}, $\alpha(f) = \lfloor |S(T)|/2 \rfloor + 1$ holds if and only if $OV_f(T)$ satisfies \eqref{eq:Jfu} and one of the following holds: (i) $L(u_0)+L(u_{p-1})=1$ and $OV_f(T)$ is admissible, except that $OV_f(T)$ is feasible when $|S(T)|$ is even; (ii) $L(u_0)+L(u_{p-1})=2$ and $OV_f(T)$ is admissible.

\medskip
\textsf{Case 2:} $|W(T)| = 2$.

We have $L(u_{0})+L(u_{p-1}) \geq 0$, and equality holds if and only if $W(T) = \{u_{0}, u_{p-1}\}$.

\medskip
\textsf{Subcase 2.1:} $|S(T)| = 1$ or $2$.

In this case, the second line in \eqref{eq:alpha} becomes $\alpha(f) \geq -(p-1)$, which is true by \eqref{eq:svp}. Moreover, in view of \eqref{eq:svp}, it is clear that $\alpha(f) = -(p-1)$ if and only if $L(u_0)+L(u_{p-1})=0$ (that is, $u_0,u_{p-1} \in W(T)$) and $\sigma(f) = -(p-1)$ (that is, $J_f(u_t,u_{t+1})=\phi(u_t,u_{t+1})=0$ and $\d(u_t,u_{t+1})=1$ for $0 \leq t \leq p-2$).

In summary, $\alpha(f) = -(p-1)$ holds if and only if $L(u_0)+L(u_{p-1})=0$ and $OV_f(T)$ is admissible and satisfies \eqref{eq:Jfu}.

\textsf{Subcase 2.2:} $|S(T)| \ge 3$.

By Lemma \ref{lem9}, we have $\alpha(f) \geq |S(T)|-(p+1)$ and so the second line in \eqref{eq:alpha} holds. Moreover, $\alpha(f) = |S(T)|-(p+1)$ if and only if the following hold: (i) when $|S(T)|$ is even, we have $L(u_0)+L(u_{p-1})=0$ (that is, $\{u_0, u_{p-1}\} \subseteq W(T)$) and $\sigma(f) = |S(T)|-(p+1)$ (that is, $\phi(u_t,u_{t+1})=0$, $\d(u_t,u_{t+1})=1$ for $0 \leq t \leq p-2$ and $J_f(T) = |S(T)|-2$), or $L(u_0)+L(u_{p-1})=2$ (that is, $\{u_0,u_{p-1}\} \subset N(W(T))$) and $\sigma(f) = |S(T)|-(p+3)$ (that is, $\phi(u_t,u_{t+1})=0$, $\d(u_t,u_{t+1})=1$ for $0 \leq t \leq p-2$ and $J_f(T) = |S(T)|-(p+3)$); (ii) when $|S(T)|$ is odd, we have $L(u_0)+L(u_{p-1})=1$ (that is, one of $u_0$ and $u_{p-1}$ is in $W(T)$ and the other is in $N(W(T))$) and $\sigma(f) = |S(T)|-(p+2)$ (that is, $\phi(u_t,u_{t+1})=0$, $\d(u_t,u_{t+1})=1$ for all $0 \leq t \leq p-2$ and $J_f(T)=|S(T)|-(p+2)$).

Thus, by the last statements in Lemmas \ref{lem8} and \ref{lem9}, $\alpha(f)=|S(T)|-(p+1)$ holds if and only if $OV_f(T)$ is admissible, satisfies \eqref{eq:Jfu}, and has the following properties: $L(u_0)+L(u_{p-1})=0$ when $|S(T)| = 1$ or $2$; $L(u_0)+L(u_{p-1})=1$ when $|S(T)| \ge 3$ is odd; $L(u_0)+L(u_{p-1}) = 0$ or $2$ when $|S(T)| \ge 4$ is even.

\medskip
Up to now we have proved that $\span(f)$ is no less than the right-hand side of \eqref{rn:lower} for any $f \in \mt{F}$. As explained earlier, this implies that the lower bound \eqref{rn:lower} is true.

\medskip
We now prove the desired necessary and sufficient condition for \eqref{rn:lower} to be tight.

\medskip
\textsf{Necessity:} Suppose that equality in \eqref{rn:lower} holds. In view of our remark right below \eqref{eq:alpha}, there exists an optimal radio labelling $f$ of $T$ such that $f \in \mt{F}$. Then $\span(f) = \rn(T) = (p-1)(d+\ve)-2L(T)+\ve+\xi$. Let $OV_f(T) : u_0,u_1,\ldots,u_{p-1}$ be the linear order of $V(T)$ induced by $f$. Set $a_t = J_f(u_t,u_{t+1})$ for $0 \leq t \leq p-2$. Since $f$ is optimal and equality in \eqref{rn:lower} holds, we have $\alpha(f) = (p-1)(\ve-1)+\ve+\xi$, or, equivalently, equality in \eqref{eq:alpha} holds. On the other hand, in the above proof of \eqref{rn:lower} we have essentially obtained a necessary condition for equality in \eqref{eq:alpha} to hold. In the case when $|W(T)|=1$ this condition can be summarized as follows: (i) if $|S(T)|$ is odd, then $L(u_0)+L(u_{p-1}) = 1$ and $OV_f(T)$ is admissible; (ii) if $|S(T)|$ is even, then either $L(u_0)+L(u_{p-1}) = 1$ and $OV_f(T)$ is feasible or admissible, or $L(u_0)+L(u_{p-1}) = 2$ and $OV_f(T)$ is admissible; (iii) for $1 \leq t \leq p-2$, $a_{t-1}+a_t = 1$ if $u_t \in S(T)$ and $\{u_{t-1},u_{t+1}\} \cap W(T) = \emptyset$, and $a_{t-1}+a_t = 0$ otherwise. In the case when $|W(T)|=2$ the necessary condition for equality in \eqref{eq:alpha} to hold can be summarized as follows: (i) $OV_f(T)$ is admissible and the following hold: $L(u_0)+L(u_{p-1})=0$ if $|S(T)| = 1$ or $2$; $L(u_0)+L(u_{p-1})=1$ if $|S(T)| \ge 3$ is odd; $L(u_0)+L(u_{p-1}) = 0$ or $2$ if $|S(T)| \ge 4$ is even; (ii) for $1 \leq t \leq p-2$, $a_{t-1}+a_t = 2$ if $u_t \in S(T)$ and $\{u_{t-1},u_{t+1}\} \cap W(T) = \emptyset$, and $a_{t-1}+a_t = 0$ otherwise.

Since $f \in \mt{F}$, we have $\phi(u_{t},u_{t+1}) = 0$, $\delta(u_{t},u_{t+1}) = 1 - \ve$, and hence $d(u_{t},u_{t+1}) = L(u_{t}) + L(u_{t+1}) + (1 - \ve)$, for $0 \leq t \leq p-2$. Plugging this into \eqref{eq:jij}, we have $f(u_{j})-f(u_{i}) = \sum_{t=i}^{j-1} \{a_{t} - (L(u_{t})+L(u_{t+1})) + (d+\ve)\}$ for $0 \le i < j \le p-1$. On the other hand, since $f$ is a radio labelling, we have $d(u_{i},u_{j}) \geq d+1- (f(u_{j})-f(u_{i}))$. Therefore,
$$
d(u_{i},u_{j}) \geq \sum_{t=i}^{j-1}\{L(u_{t})+L(u_{t+1}) - a_{t} - (d+\ve)\} + (d+1)
$$
for $0 \le i < j \le p-1$, as desired in \eqref{eqn:duv}.

\medskip
\textsf{Sufficiency:} Suppose that there exist a linear order $OV(T): u_{0},u_{1},\ldots,u_{p-1}$ of the vertices of $T$ and non-negative integers $a_0, a_1, \ldots, a_{p-2}$ such that conditions (a) and (b) are satisfied. Using these conditions, one can easily verify that $L(u_0)+L(u_{p-1}) = \ve$ and $\sum_{t=0}^{p-2} a_{t} = \xi$. Define $f: V(T) \rightarrow \{0, 1, 2, \ldots\}$ by $f(u_0) = 0$ and $f(u_{i+1}) = f(u_{i}) - (L(u_{i}) + L(u_{i+1})) + a_{i} + (d + \ve)$ for $0 \leq i \leq p-2$. Then $f(u_{j})-f(u_{i}) = \sum_{t=i}^{j-1}\{-(L(u_{t})+L(u_{t+1})) + a_{t} + (d+\ve)\} \ge (d+1)-d(u_{i},u_{j})$ for $0 \leq i < j \leq p-1$. Hence $f$ is a radio labelling of $T$. Since $L(u_0)+L(u_{p-1})+\sum_{t=0}^{p-2}a_t = \ve + \xi$, we have
\bean \span(f) & = & f(u_{p-1}) - f(u_{0}) \\
& = & -\sum_{t=0}^{p-2}\{L(u_{t})+L(u_{t+1}) - a_{t} - (d+\ve)\}
\eean
\bean
& = & (p-1)(d+\ve)-2L(T)+L(u_{0})+L(u_{p-1})+\sum_{t=0}^{p-2} a_{t} \\
& = & (p-1)(d+\ve)-2L(T)+\ve+\xi.
\eean
Hence $\span(f)$ is equal to the lower bound in \eqref{rn:lower}. Therefore, $f$ is an optimal radio labelling of $T$ and equality in \eqref{rn:lower} holds.
\end{proof}

\subsection{Proof of Theorem \ref{thm:cat}}
\label{subsec:cat}

\begin{proof} [Proof of Theorem \ref{thm:cat}]
The order, total level and parameter $\xi$ of $C(n,k)$ are given, respectively, by
$$
p = n+2k\min\{2, \lfloor (n-1)/2 \rfloor\}
$$
$$
L(C(n,k)) =
\begin{cases}
2(1+2k), & \mbox{if $n = 3,4$} \\
\frac{n^2-1}{4}+(n+5)k, & \mbox{if $n > 3$ is odd} \\
\frac{n(n-2)}{4}+(n+4)k, & \mbox{if $n > 4$ is even}
\end{cases}
$$
$$
\xi(C(n,k)) =
\begin{cases}
k, & \mbox{if $n$ is odd} \\
2(k-1), & \mbox{if $n$ is even}.
\end{cases}
$$
Plugging all these into \eqref{rn:lower}, we obtain that the right-hand side of \eqref{rn:cat} is a lower bound for $\rn(C(n,k))$. We now prove that $\rn(C(n,k))$ is equal to the right-hand side of \eqref{rn:cat}. For this purpose, it suffices to prove the existence of a linear order $u_0,u_1,\ldots,u_{p-1}$ of the vertices of $C(n,k)$ such that conditions (a) and (b) in Theorem \ref{thm:lower} are satisfied. Note that once such a linear order is determined we then obtain a unique sequence of non-negative integers $a_0, a_1, \ldots, a_{p-2}$ satisfying \eqref{eq:a0} and \eqref{eq:seq} (see Remark \ref{rem:seq}). So in what follows we will freely use these integers in our proof of condition \eqref{eqn:duv}.

Recall that the vertices of spine of $C(n,k)$ are denoted as $v_1, v_2, \ldots, v_n$, where $v_i$ is adjacent to $v_{i+1}$ for $1 \leq i \leq n-1$. Denote the $k$ neighbours of $v_i$ by $v_{i,j}$, $1 \leq j \leq k$, for $i \in \{1, (n-1)/2, (n+3)/2, n\}$ when $n$ is odd and $i \in \{1, (n-2)/2, (n+4)/2, n\}$ when $n$ is even.

\medskip
\textsf{Case 1:} $n$ is odd.

In this case, we have $W(C(n,k)) = \{v_{(n+1)/2}\}$. We consider the following two subcases separately.

\medskip
\textsf{Subcase 1.1:} $n = 3$.

Define a linear order $u_0,u_1,\ldots,u_{p-1}$ of $V(C(n,k))$ as follows. Set $u_0 = v_{(n+1)/2}$, $u_{p-2} = v_3$ and $u_{p-1} = v_1$. Set $u_t = v_{i,j}$, where
\begin{equation*}
t =
\begin{cases}
2j-1, & \mbox{if $i = n$} \\
2j, & \mbox{if $i = 1$}
\end{cases}
\end{equation*}
which ranges from $1$ to $p-3$. Then $u_{p-1} \in W(C(n,k))$ and $u_0 \in N(u_{p-1})$. Hence $L(u_0)+L(u_{p-1})=1$. Note that one remote vertex is ordered immediately after the weight center and the remaining remote vertices are ordered in an interval of even length. Hence the linear order $u_0,u_1,\ldots,u_{p-1}$ is admissible. Thus condition (a) in Theorem \ref{thm:lower} is satisfied. Observe that $u_i$ and $u_{i+1}$ are in different branches. So $\phi(u_i,u_{i+1})= 0$ for $0 \leq i \leq p-2$. It is straightforward to verify that $d(u_i,u_j)$ satisfies \eqref{eqn:duv} for $0 \leq i < j \leq p-1$, and therefore condition (b) in Theorem \ref{thm:lower} is also satisfied.

\medskip
\textsf{Subcase 1.2:} $n \geq 5$.

Define a linear order $u_0,u_1,\ldots,u_{p-1}$ of $V(C(n,k))$ as follows. Set $u_0 = v_{(n-1)/2}$ and $u_{p-1} = v_{(n+1)/2}$. Set $u_t = v_{i,j}$, where
\begin{equation*}
t =
\begin{cases}
4(j-1)+2, & \mbox{if $i = 1$} \\
4j, & \mbox{if $i = (n-1)/2$} \\
4(j-1)+3, & \mbox{if $i = (n+3)/2$} \\
4(j-1)+1, & \mbox{if $i = n$}.
\end{cases}
\end{equation*}
Note that $t$ ranges from $1$ to $4k$.  Set $u_t = v_i$, where
\begin{equation*}
t =
\begin{cases}
4k+2i, & \mbox{if $i < (n-1)/2$} \\
4k+2(i-(m+1)/2)-1, & \mbox{if $i > (n+1)/2$},
\end{cases}
\end{equation*}
which ranges from $4k+1$ to $p-2$. Then $u_{p-1} \in W(C(n,k))$ and $u_0 \in N(u_{p-1})$. Hence $L(u_0)+L(u_{p-1})=1$. Moreover, all maximum intervals of remote vertices are of even length. So the linear order $u_0,u_1,\ldots,u_{p-1}$ is feasible and thus condition (a) in Theorem \ref{thm:lower} is satisfied. Observe that, for $0 \leq i \leq p-2$, $u_i$ and $u_{i+1}$ are in different branches and hence $\phi(u_i,u_{i+1}) = 0$. The following claim implies that condition (b) in Theorem \ref{thm:lower} is satisfied as well.

\medskip
\textsf{Claim 1:} The linear order $u_0,u_1,\ldots,u_{p-1}$ defined above satisfies condition \eqref{eqn:duv}.

Consider any two vertices $u_i,u_j$ with $0 \leq i < j \leq p-1$. If $j = i+1$, then it is straightforward to verify that \eqref{eqn:duv} is satisfied. Hence we assume that $j-i \geq 2$. Denote the right-hand side of \eqref{eqn:duv} by $S_{i,j} = \sum_{t=i}^{j-1}(L(u_t)+L(u_{t+1})-a_t)-(j-i-1)(d+1)$. Note that if $j-i \geq 3$, then $S_{i,j} \leq 1 \leq d(u_i,u_j)$ and hence \eqref{eqn:duv} is satisfied. It remains to prove \eqref{eqn:duv} when $j = i+2$. So assume that $j=i+2$. Consider first the case when $0 \le i < j=i+2 \leq 4k$. If $u_i = u_0$, then $d(u_i,u_j) = d/2-1$ and $S_{i,j} = d/2-1 = d(u_i,u_j)$. If $u_i = u_{4t+1}\;(0 \leq t \leq k-1)$, then $d(u_i,u_j) = d/2$ and $S_{i,j} = d/2 = d(u_i,u_j)$. If $u_i = u_{4t+2}\;(0 \leq t \leq k-1)$, then $d(u_i,u_j) = d/2 \geq 3$ and $S_{i,j} \leq 3 \leq d(u_i,u_j)$. If $u_i = u_{4t+3}\;(0 \leq t \leq k-2)$, then $d(u_i,u_j) = d/2 \geq 3$ and $S_{i,j} \leq 2 < d(u_i,u_j)$. If $u_i = u_{4t}\;(1 \leq t \leq k-1)$, then $d(u_i,u_j) = d/2 \geq 3$ and $S_{i,j} = d/2 = d(u_i,u_j)$. Thus, if $0 \le i < j = i+2 \leq 4k$, then $d(u_i,u_j)$ satisfies \eqref{eqn:duv}. Observe that $S_{i,j} \leq 0 < d(u_i,u_j)$. Thus, in the case when $0 \leq i \leq 4k$ and $4k+1 \leq j = i+2 \leq p-1$, $d(u_i,u_j)$ also satisfies \eqref{eqn:duv}. Finally, if $4k+1 \leq i < j = i+2 \leq p-1$, then $S_{i,j} = 0 < d(u_i,u_j)$ and hence $d(u_i,u_j)$ satisfies \eqref{eqn:duv}. This completes the proof of Claim 1.

\medskip
\textsf{Case 2:} $n$ is even.

In this case, we have $W(C(n,k)) = \{v_{n/2},v_{n/2+1}\}$. We consider the following two subcases  separately.

\medskip
\textsf{Subcase 2.1:} $n = 4$.

Define a linear order $u_0,u_1,\ldots,u_{p-1}$ of $V(C(n,k))$ as follows. Set $u_0 = v_{n/2}$, $u_{p-3} = v_4$, $u_{p-2} = v_1$ and $u_{p-1} = v_{n/2+1}$. Set $u_t = v_{i,j}$, where
\begin{equation*}
t =
\begin{cases}
2j-1, & \mbox{if $i = n$} \\
2j, & \mbox{if $i = 1$},
\end{cases}
\end{equation*}
which ranges from $1$ to $p-4$. Then $L(u_0)+L(u_{p-1})=0$. Note that one remote vertex is ordered immediately after a weight center and the remaining remote vertices are ordered in an interval of even length. So the linear order $u_0,u_1,\ldots,u_{p-1}$ defined this way is admissible, and thus condition (a) in Theorem \ref{thm:lower} is satisfied. Observe that $u_i$ and $u_{i+1}$ are in opposite branches for all $0 \leq i \leq p-2$ and hence $\phi(u_i,u_{i+1}) = 0$ and $\delta(u_i,u_{i+1}) = 1$. It is straightforward to verify that $d(u_i,u_j)$ satisfies \eqref{eqn:duv} for $0 \leq i < j \leq p-1$. That is, condition (b) in Theorem \ref{thm:lower} is satisfied.

\medskip
\textsf{Subcase 2.2:} $n \geq 6$.

Define a linear order $u_0,u_1,\ldots,u_{p-1}$ of $V(C(n,k))$ as follows. Set $u_0 = v_{n/2-1}$, $u_1 = v_{n,1}$, $u_2 = v_{n/2}$, $u_3 = v_{n,2}$, $u_4 = v_{1,1}$, $u_5 = v_{n/2+1}$, $u_6 = v_{1,2}$ and $u_{p-1} = v_{n/2+2}$. Set $u_t = v_{i,j}$, where
\begin{equation*}
t =
\begin{cases}
4(j-1)+2, & \mbox{if $i = 1$ and $3 \leq j \leq k$}, \\
4(j+1), & \mbox{if $i = n/2-1$ and $1 \leq j \leq k-1$}, \\
4(j+1)-1, & \mbox{if $i = n/2+2$ and $1 \leq j \leq k-1$}, \\
4(j-1)+1, & \mbox{if $i = n$ and $3 \leq j \leq k$}.
\end{cases}
\end{equation*}
Note that this $t$ ranges from $7$ to $4k$. Set $u_{4k+1} = v_{n/2+2,k}$ and $u_{4k+2} = v_{n/2-1,k}$. Set $u_t = v_i$, where
\begin{equation*}
t =
\begin{cases}
4k+2(n/2-i), & \mbox{if $i < n/2-1$} \\
4k+2(n-i)+3, & \mbox{if $i > n/2+2$},
\end{cases}
\end{equation*}
which ranges from $4k+3$ to $p-2$. Then $L(u_0)+L(u_{p-1})=2$ and $W(T) = \{u_2, u_5\}$. We have $\{u_1,u_3,u_4,u_6\} \subseteq S(T)$ and all other remote vertices are ordered in intervals of even lengths. Hence the linear order $u_0,u_1,\ldots,u_{p-1}$ defined this way is admissible. Thus condition (a) in Theorem \ref{thm:lower} is satisfied. Observe that, for $0 \leq i \leq p-2$, $u_i$ and $u_{i+1}$ are in opposite branches and hence $\phi(u_i,u_{i+1}) = 0$ and $\delta(u_i,u_{i+1}) = 1$. Now we prove that condition (b) in Theorem \ref{thm:lower} is also satisfied. To this end it suffices to prove the following claim under the condition that $a_{i-1}+a_i$ is equal to $2$ if $u_i \in S(T)$ and $\{u_{i-1}, u_{i+1}\} \cap W(T) = \emptyset$ and 0 otherwise.

\medskip
\textsf{Claim 2:} The linear order $u_0,u_1,\ldots,u_{p-1}$ defined above satisfies \eqref{eqn:duv}.

Consider any two vertices $u_i$ and $u_j$ with $0 \leq i < j \leq p-1$. If $j=i+1$, then it is straightforward to show that \eqref{eqn:duv} is satisfied. Hence we assume that $j \geq i+2$. Denote the right-hand side of \eqref{eqn:duv} by $S_{i,j} = \sum_{t=i}^{j-1}(L(u_t)+L(u_{t+1})-a_t)-(j-i)d+(d-1)$. If $j-i \geq 4$, then $S_{i,j} \leq 1 \leq d(u_i,u_j)$ and hence \eqref{eqn:duv} is satisfied. If $j = i+3$, then $d(u_i,u_j) \geq 2$ and hence $S_{i,j} \leq 2 \leq d(u_i,u_j)$. Hence \eqref{eqn:duv} is satisfied when $j = i+3$. It remains to prove that \eqref{eqn:duv} is satisfied when $j = i+2$. Assume that $j=i+2$. It is straightforward to verify that $d(u_i,u_j)$ satisfies \eqref{eqn:duv} when $0 \leq i < j \leq 6$ or $0 \leq i \leq 6$ and $7 \leq j \leq 4k+2$. Assume that $7 \leq i < j \leq 4k+2$. If $u_i = u_{4t+3}\;(1 \leq t \leq k-1)$, then $d(u_i,u_j) = (d+3)/2$ and $S_{i,j} < 3 \leq d(u_i,u_j)$. If $u_i = u_{4(t+1)}\;(1 \leq t \leq k-1)$, then $d(u_i,u_j) = (d-1)/2$ and $S_{i,j} = (d-3)/2 < (d-1)/2$. If $u_i = u_{4t+1}\;(2 \leq t \leq k)$, then $d(u_i,u_j) = (d-1)/2$ and $S_{i,j} = (d-5)/2 < (d-1)/2$. If $u_i = u_{4t+2}\;(2 \leq t \leq k)$, then $d(u_i,u_j) \geq 3$ and $S_{i,j} < 2 \leq d(u_i,u_j)$. So $d(u_i,u_j)$ satisfies \eqref{eqn:duv} when $u_i,u_j \in U_2$. If $7 \leq i \leq 4k+2$ and $4k+3 \leq j \leq p-1$, then $S_{i,j} \leq 2 \leq d(u_i,u_j)$ and hence $d(u_i,u_j)$ satisfies \eqref{eqn:duv}. Finally, if $4k+3 \leq i < j \leq p-1$, then $S_{i,j} = 1 \leq d(u_i,u_j)$ and hence $d(u_i,u_j)$ satisfies \eqref{eqn:duv}. Therefore, in all possibilities we have proved that $d(u_i,u_j)$ satisfies \eqref{eqn:duv}. This completes the proof of Claim 2.

In summary, in each case above we have given a linear order $u_0,u_1,\ldots,u_{p-1}$ of the vertices of $C(n,k)$ such that conditions (a) and (b) in Theorem \ref{thm:lower} are satisfied. Thus, by Theorem \ref{thm:lower}, $\rn(C(n,k))$ is equal to the right-hand side of \eqref{rn:lower}, which is exactly the right-hand side of \eqref{rn:cat} in this special case.
\end{proof}

It is worth mentioning that by Theorem \ref{thm:lower} the labelling defined in \eqref{f00} and \eqref{f11} with underlying linear order (and associated non-negative integers) given in the proof above is an optimal radio labelling of $C(n,k)$. We illustrate this for $C(5,3)$ and $C(6,3)$ in Figure \ref{Cater}.

\begin{figure}[ht]
\begin{center}
\includegraphics[width=5.8 in]{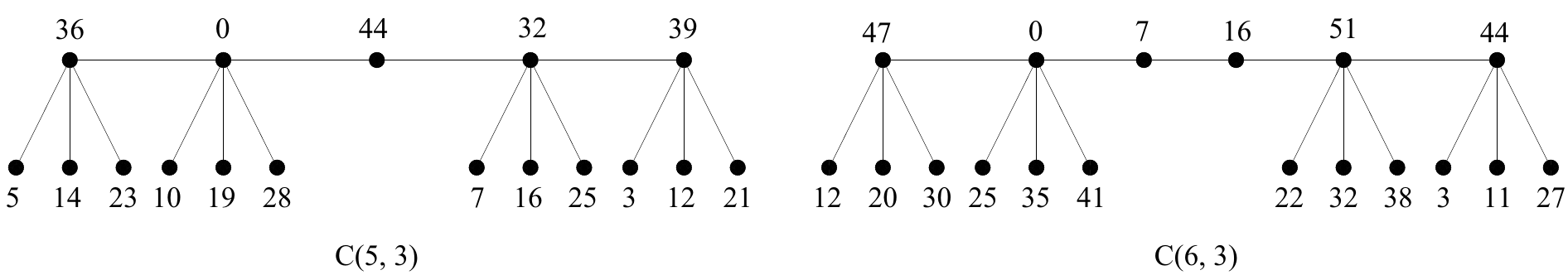}
\caption{Optimal radio labellings of $C(5,3)$ and $C(6,3)$ obtained from the proof of Theorem \ref{thm:ds}.}
\label{Cater}
\end{center}
\end{figure}

\section{Radio number of level-wise regular two-branch trees}
\label{level}

The \emph{eccentricity} of a vertex in a graph is the maximum distance from it to any vertex in the graph \cite{West}. The \emph{center} of a tree comprises all vertices with minimum eccentricity \cite{West}. It is well known that the center $L_{0}$ of any tree $T$ consists of one vertex $r$ or two adjacent vertices $r_1, r_2$, that is, $L_{0} = \{r\}$ or $L_{0} = \{r_1, r_2\}$, depending on whether $\diam(T)$ is even or odd. We may think of $T$ as rooted at $L_0$. Call $d(u, L_{0}) = \min\{d(u, v): v \in L_{0}\}$ the \emph{distance} from $u \in V(T)$ to $L_{0}$ and $h = \lfloor \diam(T)/2 \rfloor$ the \emph{height} of $T$. (We defined the ``weight center" of a tree, the ``level of a vertex" in a tree and the ``total level" of a tree in Section \ref{prel}. Note the difference between these terms and the ``center" of a tree, the ``eccentricity of a vertex" in a tree and the ``height" of a tree, respectively.) Then $d(u, L_{0}) \in \{0, 1, \ldots, h\}$ for all $u \in V(T)$. Define
$$
L_{i} = \{u \in V(T): d(u, L_{0}) = i\},\ 1 \leq i \leq h.
$$
The vertices in $L_{i}$ are said to be at level $i$, for $0 \leq i \leq h$. We say that $T$ is a  \emph{level-wise regular tree} \cite{Tuza} if, for $0 \leq i \leq h$, all vertices in $L_i$ have the same degree, say, $m_i$. Note that $m_{h} = 1$ and all leaves of a level-wise regular tree $T$ are in $L_{h}$. Note also that, up to isomorphism, $T$ is uniquely determined by $(m_{0},m_{1},\ldots,m_{h-1})$ and $z = |L_{0}| \in \{1, 2\}$. Following \cite{Tuza}, we use $T^{z}_{m_{0},m_{1},\ldots,m_{h-1}}$ to denote this tree in the sequel. It can be shown that for $T^{z}_{m_{0},m_{1},\ldots,m_{h-1}}$ the center and weight center are identical, and $T^{z}_{m_{0},m_{1},\ldots,m_{h-1}}$ is a two-branch tree precisely when $m_{0} = 2$.

In \cite[Theorem 4]{Tuza}, Hal\'{a}sz and Tuza obtained a lower bound for $\rn(T^{z}_{m_{0},m_{1},\ldots,m_{h-1}})$, $z=1,2$, for any $h \ge 1$ and $m_{0},m_{1},\ldots,m_{h-1} \ge 2$, and in \cite[Theorem 5]{Tuza} they proved further that this lower bound is tight when all $m_{i} \geq 3$ for $0 \le i \le h-1$. However, if not all $m_{i}$ for $0 \le i \le h-1$ are no less than $3$, then the value of $\rn(T^{z}_{m_{0},m_{1},\ldots,m_{h-1}})$ is unknown in general. The next theorem covers the case when $m_0 = 2$ and all other $m_{i}$ for $1 \le i \le h-1$ are no less than $3$. Note that this result cannot be obtained from \cite[Theorem 4]{Tuza} since the value of $\rn(T^{z}_{m_{0},m_{1},\ldots,m_{h-1}})$ in this case as given below is significantly larger than the lower bound in \cite[Theorem 4]{Tuza}.

\begin{Theorem}
\label{thm:level}
Let $h \geq 1$ and $m_{i} \geq 3$, for $1 \leq i \leq h-1$. Set $T^{z} = T^{z}_{2,m_{1},\ldots,m_{h-1}}$ for $z=1,2$. Then
\begin{equation}
\label{rn:level}
\rn(T^{z}) =
\begin{cases}
\sum_{i=1}^{h-1}\(4(h-i)-2\)\(\prod_{j=1}^{i}(m_{j}-1)\)+\prod_{i=1}^{h-1}(m_{i}-1)+4h-1, & \mbox{if  } z = 1 \\
\sum_{i=1}^{h-1}\(4(h-i)-2\)\(\prod_{j=1}^{i}(m_{j}-1)\)+2\prod_{i=1}^{h-1}(m_{i}-1)+6h-3, & \mbox{if  } z = 2.
\end{cases}
\end{equation}
\end{Theorem}

\begin{proof}
The order, total level and parameter $\xi$ of $T^{z}$ are given, respectively, by
$$
p(T^{z}) =
\begin{cases}
3 + 2\sum_{i=1}^{h-1}\(\prod_{j=1}^{i}(m_{j}-1)\), & \mbox{if  } z = 1 \\
4 + 2\sum_{i=1}^{h-1}\(\prod_{j=1}^{i}(m_{j}-1)\), & \mbox{if  } z = 2
\end{cases}
$$
$$
L(T^{z}) = 2 + 2\sum_{i=1}^{h-1}(i+1)\(\prod_{j=1}^{i}(m_{j}-1)\)
$$
and
$$
\xi(T^{z}) =
\begin{cases}
\prod_{i=1}^{h-1}(m_{i}-1), & \mbox{if } z=1 \\
2\prod_{i=1}^{h-1}(m_{i}-1)-2, & \mbox{if } z=2.
\end{cases}
$$
Plugging these into \eqref{rn:lower}, we obtain that the right-hand side of \eqref{rn:level} is a lower bound for $\rn(T^{z})$. Now we prove that this lower bound is the actual value of $\rn(T^{z})$. To this end it suffices to find a linear order $u_{0},u_{1},\ldots,u_{p-1}$ of the vertices of $T^{z}$ such that conditions (a) and (b) in Theorem \ref{thm:lower} are satisfied. Once such a linear order is found, we then obtain a corresponding sequence of non-negative integers $a_0, a_1, \ldots, a_{p-2}$ satisfying \eqref{eq:seq} (see Remark \ref{rem:seq}). These are the integers which will be used in our proof of condition \eqref{eqn:duv}.

\medskip
\textsf{Case 1:} $z$ = 1.

Let $w$ be the unique weight center of $T^{1}$. Denote the two children of $w$ by $w_{0}$ and $w_{1}$. Denote the $m_{1}-1$ children of $w_{0}$ and $w_{1}$ by $w_{00},w_{01},\ldots,w_{0(m_{1}-2)}$ and $w_{10},w_{11},\ldots,w_{1(m_{1}-2)}$, respectively. Inductively, denote the $m_{l-1}$ children of $w_{i_{1}i_{2} \ldots i_{l}}$ ($0 \leq i_{1} \leq 1, 0 \leq i_{2} \leq m_{1}-2,\ldots,0 \leq i_{l} \leq m_{l-1}-2$) by $w_{i_{1}i_{2} \ldots i_{l}i_{l+1}}$ ($0 \leq i_{l+1} \leq m_{l}-2$). Set $u_{0}$ = $w$. Set $u_{j} = w_{i_{1}i_{2} \ldots i_{l}}$, where
\be
\label{eq:j}
j = 1 + i_{1} + m_{0} \sum_{t=2}^l i_{t}(m_{1}-1) \cdots (m_{t-2}-1) + m_{0} \sum_{t = l+1}^{h-1} (m_{1}-1) \cdots (m_{t}-1)
\ee
which ranges from $1$ to $p-1$. Then $u_{p-1}$ is adjacent to $u_{0} = w$ and hence $L(u_{0})+L(u_{p-1})$ = 1. Moreover, one remote vertex is immediately after the weight center $w$ and all other remote vertices are ordered in interval of even length except possibly one. Thus condition (a) in Theorem \ref{thm:lower} is satisfied. Note that $u_{i}$ and $u_{i+1}$ are in different branches and so $\phi(u_{i},u_{i+1})$ = 0 for $1 \leq i \leq p-2$. The following claim shows that condition (b) in Theorem \ref{thm:lower} is also satisfied.

\medskip
\textsf{Claim 1:} The linear order $u_{0},u_{1},\ldots,u_{p-1}$ defined above satisfies condition \eqref{eqn:duv}.

Consider any pair of vertices $u_{i}, u_{j}$ with $0 \leq i < j \leq p-1$. Let $L(u_{i}) = l_{i}$ and $L(u_{j}) = l_{j}$. Denote $A = a_{i}+a_{i+1}+\ldots+a_{j}$. Then $l_{i} > l_{j}$ as $i < j$. The right-hand side of \eqref{eqn:duv} is equal to $S_{i,j} = \sum_{t=i}^{j-1}\(L(u_{t})+L(u_{t+1})-a_t\) - (j-i-1)(d+1) \leq (j-i)l_{i}+(j-i-1)l_{i} + l_{j} - A - (j-i-1)(d+1) = l_{i}+l_{j}-A-(j-i-1)(d+1-2l_{i})$. If $u_{i}$ and $u_{j}$ are in different branches, then $d(u_{i},u_{j}) = l_{i}+l_{j}$ and so $S_{i,j} \leq l_{i}+l_{j} = d(u_{i},u_{j})$ as $j-i \geq 1$, $d-2l_{i} \geq 0$ and $A \geq 0$. If $u_{i}$ and $u_{j}$ are in the same branch, then $d(u_{i},u_{j}) = l_{i}+l_{j}-2\phi(u_{i},u_{j})$ with $\phi(u_{i},u_{j}) \geq 1$ and $j-i = 2m$ for some $m \geq 1$. It is clear from the definition of the linear order above that $m \geq \phi(u_{i},u_{j})$. If $u_{i}, u_{j} \in S(T)$, then $S_{i,j} \leq l_{i}+l_{j}-((j-i)/2)-(j-i-1) \leq l_{i}+l_{j}-(j-i) \leq l_{i}+l_{j}-2m \leq l_{i}+l_{j}-2\phi(u_{i},u_{j}) = d(u_{i},u_{j})$. If $u_{i} \in S(T)$ but $u_{j} \not\in S(T)$, then there exists $k$ with $i \leq k \leq j$ such that $u_{k} \in S(T)$ and $u_{k+1} \not\in S(T)$. Let $L(u_{k}) = l_{k}$ and $A' = a_{i}+a_{i+1}+\ldots+a_{k}$. Then $S_{i,j} = S_{i,k}+S_{k,j}-(d+1) \leq l_{i}+l_{k}-A'-(k-i-1)(d+1-2l_{i})+l_{k}+l_{j}-(j-k-1)(d+1-2(l_{k}-1))-(d+1) \leq l_{i}+l_{k}-((k-i)/2)-(k-i-1)+l_{k}+l_{j}-2(j-k-1)-(d+1) \leq l_{i}+l_{k}-(k-i)+l_{k}+l_{j}-(j-k)-(d+1) \leq l_{i}+l_{j}-(j-i)-(d+1-2l_{k}) \leq l_{i}+l_{j}-(j-i) \leq l_{i}+l_{j}-2m \leq l_{i}+l_{j}-2\phi(u_{i},u_{j}) = d(u_{i},u_{j})$. If $u_{i}, u_{j} \not\in S(T)$, then $S_{i,j} \leq l_{i}+l_{j}-(j-i-1)(d+1-2l_{i}) \leq l_{i}+l_{j}-2(j-i-1) \leq l_{i}+l_{j}-(j-i) \leq l_{i}+l_{j}-2m \leq l_{i}+l_{j}-2\phi(u_{i},u_{j}) = d(u_{i},u_{j})$. This completes the proof of Claim 1.

So far we have proved that the linear order $u_0,u_1,\ldots,u_{p-1}$ of the vertices of $T^1$ determined by \eqref{eq:j} satisfies conditions (a) and (b) in Theorem \ref{thm:lower}. Hence, by Theorem \ref{thm:lower}, $\rn(T^1)$ is equal to the right-hand side of \eqref{rn:lower} which is exactly the first line on the right-hand side of \eqref{rn:level}.

\medskip
\textsf{Case 2:} $z = 2$.

Let $w$ and $w'$ be  the weight centers of $T^{2}$. Denote the unique child of $w$ and $w'$ by $w_{0}$ and $w'_{0}$, respectively. Denote the $m_{1}-1$ children of $w_{0}$ by $w_{00},w_{01},\ldots,w_{0(m_{1}-2)}$ and the $m_{1}-1$ children of $w'_{0}$ by $w'_{00},w'_{01},\ldots,w'_{0(m_{1}-2)}$. Inductively, denote the $m_{l-1}$ children of $w_{i_{1}i_{2} \ldots i_{l}}$ and $w'_{i_{1}i_{2} \ldots i_{l}}$ ($0 \leq i_{1} \leq 1, 0 \leq i_{2} \leq m_{1}-2,\ldots,0 \leq i_{l} \leq m_{l-1}-2$) by $w_{i_{1}i_{2} \ldots i_{l}i_{l+1}}$ and $w'_{i_{1}i_{2} \ldots i_{l}i_{l+1}}$, respectively, where $0 \leq i_{l+1} \leq m_{l}-2$. Set $v_{j}$ = $w_{i_{1}i_{2} \ldots i_{l}}$ and $v'_{j} = w'_{i_{1}i_{2} \ldots i_{l}}$, where $j$ is as given in \eqref{eq:j}. Let $u_{0} = v_{(p-2)/2}$, $u_{1} = v'_{1}$, $u_{2} = w$, $u_{3} = v'_{2}$, $u_{4} = v_{1}$, $u_{5} = w'$, $u_{6} = v_{2}$ and $u_{p-1} = v'_{(p-2)/2}$. For $7 \leq j \leq p-2$, let
\begin{equation*}
u_{j} =
\begin{cases}
v_{\frac{j-2}{2}}, & \mbox{if $j$ is even}, \\
v'_{\frac{j-1}{2}}, & \mbox{if $j$ is odd}.
\end{cases}
\end{equation*}
Then $u_{0}$ and $u_{p-1}$ are adjacent to the weight centers and hence $L(u_{0})+L(u_{p-1}) = 2$. Moreover, the four remote vertices are immediately before and after the weight centers and all other remote vertices are ordered in an interval of length even. Hence the linear order $u_0,u_1,\ldots,u_{p-1}$ thus defined is admissible. So condition (a) in Theorem \ref{thm:lower} is satisfied. Note that, for $0 \leq i \leq p-2$, we have $\phi(u_{i},u_{i+1}) = 0$ and $\delta(u_{i},u_{i+1}) = 1$ as $u_{i}$ and $u_{i+1}$ are in opposite branches. The following claim implies that condition (b) in Theorem \ref{thm:lower} is also satisfied.

\medskip
\textsf{Claim 2:} The linear order $u_{0},u_{1},\ldots,u_{p-1}$ defined above satisfies condition \eqref{eqn:duv}.

Consider any pair of vertices $u_{i},u_{j}$ with $0 \leq i < j \leq p-1$. It is easy to verify that condition \eqref{eqn:duv} is satisfied when $0 \leq i < j \leq 7$, or $0 \leq i \leq 7 < j$, or $0 \leq i \leq p-2$ and $j = p-1$. Hence we assume $7 \leq i < j \leq p-2$ in the sequel. Let $L(u_{i}) = l_{i}$ and $L(u_{j}) = l_{j}$. Set $A = a_{i}+a_{i+1}+\ldots+a_{j}$. Then $l_{i} > l_{j}$ as $i < j$. The right-hand side of \eqref{eqn:duv} is equal to $S_{i,j} = \sum_{t=i}^{j-1}\(L(u_{t})+L(u_{t+1})-a_t\)-(j-i)d+(d+1) \leq (j-i)l_{i}+(j-i-1)l_{i}+l_{j}-A-(j-i-1)d+1 =  l_{i}+l_{j}-A-(j-i-1)(d-2l_{i})+1$. If $u_{i}$ and $u_{j}$ are in opposite branches, then $d(u_{i},u_{j}) = l_{i}+l_{j}+1$ and so $S_{i,j} \leq l_{i}+l_{j}+1 = d(u_{i},u_{j})$ as $j-i \geq 1$, $d-2l_{i} \geq 1$ and $A \geq 0$. If $u_{i}$ and $u_{j}$ are in the same branch, then $d(u_{i},u_{j}) = l_{i}+l_{j}-2\phi(u_{i},u_{j})$ with $\phi(u_{i},u_{j}) \geq 1$ and $j-i = 2m$ for some $m \geq 1$. It is clear from the definition of the linear order above that $m \geq \phi(u_{i},u_{j})$. If $u_{i}, u_{j} \in S(T)$, then $S_{i,j} \leq l_{i}+l_{j}-A-(j-i-1)(d-2l_{i})+1 \leq l_{i}+l_{j}-(j-i)-(j-i-1)(d-2l_{i})+1 = l_{i}+l_{j}-(j-i-1)(d-2l_{i}+1) \leq l_{i}+l_{j}-2(j-i-1) \leq l_{i}+l_{j}-(j-i) \leq l_{i}+l_{j}-2m \leq l_{i}+l_{j}-2\phi(u_{i},u_{j}) = d(u_{i},u_{j})$. If $u_{i} \in S(T)$ but $u_{j} \not\in S(T)$, then there exists $k$ with $i \leq k \leq j$ such that $u_{k} \in S(T)$ and $u_{k+1} \not\in S(T)$. Let $L(u_{k}) = l_{k}$ and $A' = a_{i}+a_{i+1}+\ldots+a_{k}$. Then $S_{i,j} = S_{i,k}+S_{k,j}-d+1 \leq l_{i}+l_{k}-A'-(k-i-1)(d-2l_{i})+l_{k}+l_{j}-(j-k-1)(d-2(l_{k}-1))-d+1 \leq l_{i}+l_{k}-(k-i)-(k-i-1)+l_{k}+l_{j}-2(j-k-1)-d+1 \leq l_{i}+l_{k}-(k-i)+l_{k}+l_{j}-(j-k)-(d-1) \leq l_{i}+l_{j}-(j-i)-(d-1-2l_{k}) \leq l_{i}+l_{j}-(j-i) \leq l_{i}+l_{j}-2m \leq l_{i}+l_{j}-2\phi(u_{i},u_{j}) = d(u_{i},u_{j})$. If $u_{i},u_{j} \not\in S(T)$, then $S_{i,j} \leq l_{i}+l_{j}-(j-i-1)(d-2l_{i}) \leq l_{i}+l_{j}-2(j-i-1) \leq l_{i}+l_{j}-(j-i) \leq l_{i}+l_{j}-2m \leq l_{i}+l_{j}-2\phi(u_{i},u_{j}) = d(u_{i},u_{j})$. This completes the proof of Claim 2.

Summing up, we have proved that the linear order $u_0,u_1,\ldots,u_{p-1}$ of the vertices of $T^2$ defined above satisfies conditions (a) and (b) in Theorem \ref{thm:lower}. Hence, by Theorem \ref{thm:lower}, $\rn(T^2)$ is equal to the right-hand side of \eqref{rn:lower} which is exactly the second line on the right-hand side of \eqref{rn:level}.
\end{proof}

The linear order in the proof above determines an optimal radio labelling of $T^z$, $z=1, 2$ through \eqref{f00} and \eqref{f11}. We illustrate this for $T_{2,4,4}^{1}$ and $T_{2,4,4}^{2}$ in Figure \ref{Fig11}.

\begin{figure}[ht]
\begin{center}
\includegraphics[width=5.8 in]{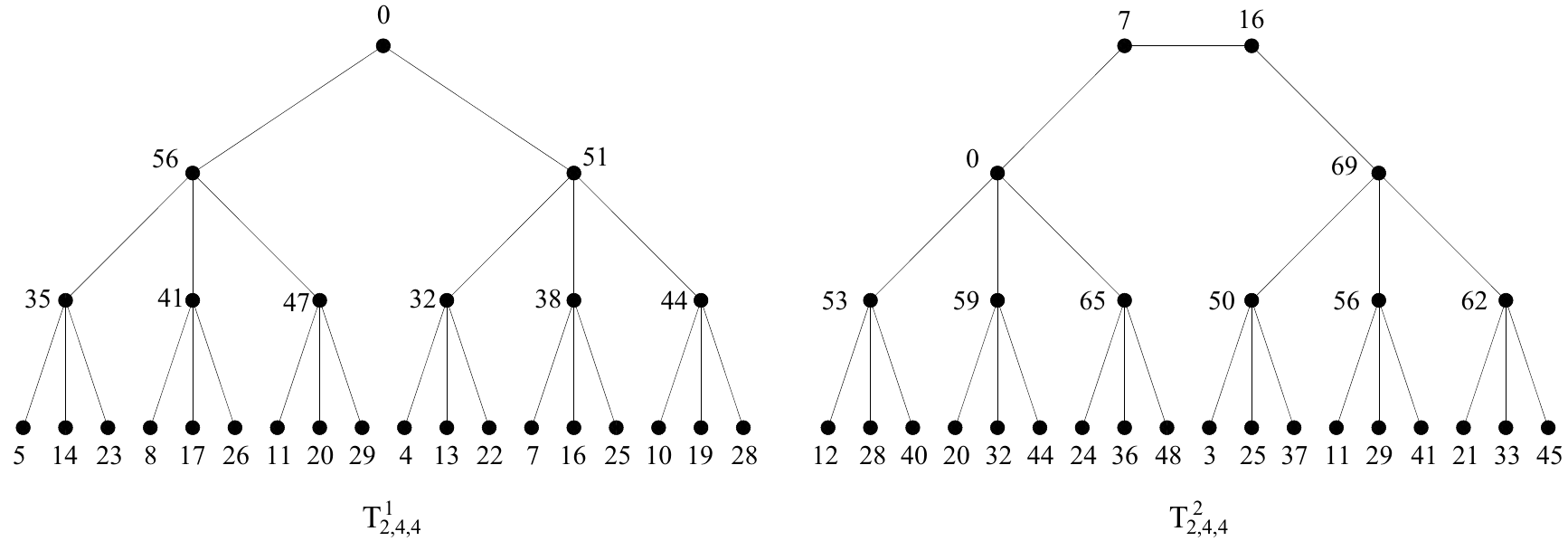}
\caption{Optimal radio labellings of $T_{2,4,4}^{1}$ and $T_{2,4,4}^{2}$ obtained from \eqref{f00} and \eqref{f11} using the linear order in the proof of Theorem \ref{thm:level}.}
\label{Fig11}
\end{center}
\end{figure}

Given integers $m \geq 2$ and $h \geq 1$, the level-wise regular tree $T_{h,m} = T^{1}_{m_{0},m_{1},\ldots,m_{h-1}}$ for which $m_{0} = m$ and $m_{1} = \cdots = m_{h-1} = m+1$ is called the \emph{complete $m$-ary tree of height $h$}. In particular, $T_{h} = T_{h,2} $ is the \emph{complete binary tree of height $h$}. In \cite[Theorem 2]{Li}, Li \emph{et al.} gave a formula for $\rn(T_{h,m})$ for any $h \geq 2$ and $m \geq 3$, and later the authors of the present paper showed that this formula can be obtained from \cite[Theorem 3.2]{Bantva2}. Another main result in \cite{Li} can be obtained from Theorem \ref{thm:level} by setting $z=1$ and $m_{1} = m_{2} = \cdots = m_{h-1} = 3$ in \eqref{rn:level}. We present this result in the following corollary.

\begin{Corollary}
\label{coro:Li}
(\cite[Theorem 1]{Li}) Let $h \ge 2$ be an integer. Then
$$
\rn(T_{h}) = 13 \cdot 2^{h-1}-4h-5.
$$
\end{Corollary}

In the case when $m_1 = m+1$ for some integer $m \geq 2$ and $m_0 = m_2 = \cdots = m_{h-1} = 2$, we write $L_{m,h}^{z}$ in place of $T^{z}_{m_{0},m_{1},\ldots,m_{h-1}}$. Recall that Theorem \ref{thm:level} cannot be proved using \cite[Theorem 4]{Tuza}. The next result handles another case which cannot be covered by \cite[Theorem 4]{Tuza}.

\begin{Theorem}\label{thm:ds} Let $m \ge 2$ and $h \geq 2$. Then
\begin{equation}\label{rn:ds}
\rn(L_{m,h}^z) =
\begin{cases}
2mh(h-2)+3m+4h-1, & \mbox{if  } z = 1 \\
2mh(h-2)+4m+6h-2, & \mbox{if  } z = 2.
\end{cases}
\end{equation}
\end{Theorem}

\begin{proof}
The order, total level and parameter $\xi$ of $L_{m,h}^z$ are given, respectively, by
$$
p =
\begin{cases}
2mh-2m+3, & \mbox{if } z = 1, \\
2mh-2m+4, & \mbox{if } z = 2
\end{cases}
$$
$$
L(L_{m,h}^z) = 2 + m(h(h+1)-2)
$$
and
$$
\xi(L_{m,h}^z) =
\begin{cases}
m, & \mbox{if } z = 1 \\
2m, & \mbox{if } z = 2.
\end{cases}
$$
Plugging these into \eqref{rn:lower}, we obtain that the right-hand side of \eqref{rn:ds} is a lower bound for $\rn(L_{m,h}^z)$. We now prove that this lower bound is the actual value of $\rn(L_{m,h}^z)$. To this end it suffices to find a linear order $u_0,u_1,\ldots,u_{p-1}$ of the vertices of $L_{m,h}^z$ such that conditions (a) and (b) in Theorem \ref{thm:lower} are satisfied. As seen in the proofs of Theorems \ref{thm:cat} and \ref{thm:level}, once such a linear order is found, we then have a corresponding sequence of non-negative integers $a_0, a_1, \ldots, a_{p-2}$ satisfying \eqref{eq:seq}. These are the integers which will be used in our proof of condition \eqref{eqn:duv}.

Let $r$ be the unique weight center of $L_{m,h}^1$, and denote the two children of $r$ by $w^1$ and $w^2$. Let $r_1$ and $r_2$ be the weight centers of $L_{m,h}^2$, and denote the child of $r_1$ and $r_2$ not in $\{r_1,r_2\}$ by $w^1$ and $w^2$, respectively. In either case, denote the $m$ children of $w^{1}$ and $w^{2}$ by $w^{1}_{1,1},w^{1}_{2,1},\ldots,w^{1}_{m,1}$ and $w^{2}_{1,1},w^{2}_{2,1},\ldots,w^{2}_{m,1}$, respectively. Inductively, denote the child of $w^{l}_{i,j}$ ($1 \leq l \leq 2, 1 \leq i \leq m, 1 \leq j \leq h-2$) by $w^{l}_{i,j+1}$. For convenience, for $z=1,2$ and $l = 1, 2$, the path of $L_{m,h}^{z}$ with length $h-2$ induced by a child of $w^l$ and all its descendants is called a leg. Two legs are said to be distinguishable if they use different children of the same $w^l$.

\medskip
\textsf{Case 1:} $z = 1$.

In this case, set $u_0 = r$, $u_{p-2} = w^{1}$ and $u_{p-1} = w^{2}$. For $1 \leq l \leq 2$ and $1 \leq i \leq m$, set $u_t = w^{l}_{i,h-1}$, where
$$
t = 2i+l-2,
$$
which ranges from $1$ to $2m$. For $1 \leq l \leq 2$, $1 \leq i \leq m$ and $1 \leq j \leq h-2$, set $u_t = w^{l}_{i,j}$, where
\begin{equation*}
t =
\begin{cases}
2(i-1)+2mj+l, & \mbox{ if }l=1  \\
2(i-1)+2m(h-j-1)+l, & \mbox{ if }l=2,
\end{cases}
\end{equation*}
which ranges from $2m+1$ to $p-3$. Then $u_{p-1} = w^{2}$ is adjacent to $u_{0} = w$ and hence $L(u_{0})+L(u_{p-1}) = 1$. Moreover, one remote vertex is immediately after the weight center $r$ and all remaining remote vertices are ordered in an interval of even length except possibly one. Hence the linear order $u_{0},u_{1},\ldots,u_{p-1}$ thus defined is admissible. Thus condition (a) in Theorem \ref{thm:lower} is satisfied. Note that, for $0 \leq t \leq p-2$, $u_{t}$ and $u_{t+1}$ are in different branches and hence $\phi(u_t,u_{t+1}) = 0$. The following claim shows that condition (b) in Theorem \ref{thm:lower} is also satisfied.

\medskip
\textsf{Claim 1:} The linear order $u_{0},u_{1},\ldots,u_{p-1}$ defined above satisfies condition \eqref{eqn:duv}.

Consider any two vertices $u_{i},u_{j}$ with $0 \leq i < j \leq p-1$. Let $L(u_{i}) = l_i$ and $L(u_{j}) = l_j$. Denote $A = a_i+a_{i+1}+\ldots+a_j$. Then $l_i > l_j$ as $i < j$. The right-hand side of \eqref{eqn:duv} is equal to $S_{i,j} = \sum_{t=i}^{j-1}\(L(u_{t})+L(u_{t+1})-a_t\)-(j-i-1)(d+1) \leq (j-i)l_{i}+(j-i-1)l_{i}+l_{j}-A-(j-i-1)(d+1) = l_{i}+l_{j}-A-(j-i-1)(d+1-2l_{i})$. If $u_{i}$ and $u_{j}$ are in different branches, then $d(u_{i},u_{j}) = l_{i}+l_{j}$ and so $S_{i,j} \leq l_{i}+l_{j} = d(u_{i},u_{j})$ as $j-i \geq 1$, $d-2l_{i} \geq 0$ and $A \geq 0$. If $u_{i}$ and $u_{j}$ are in the same branch, then $d(u_{i},u_{j}) = l_{i}+l_{j}-2\phi(u_{i},u_{j})$ with $\phi(u_{i},u_{j}) \geq 1$. If $u_{i},u_{j} \in S(T)$, then $\phi(u_{i},u_{j}) = 1$ and $S_{i,j} \leq l_i+l_j-((j-i)/2)-(j-i-1) \leq l_i+l_j-(j-i) \leq l_i+l_j-2 = l_i+l_j-2\phi(u_i,u_j) = d(u_i,u_j)$. If $u_i \in S(T)$ but $u_j \not\in S(T)$, then there exists $k$ with $i \leq k \leq j$ such that $u_{k} \in S(T)$ and $u_{k+1} \not\in S(T)$. Let $L(u_{k}) = l_k$ and $A' = a_i+a_{i+1}+\ldots+a_k$. Then $S_{i,j} = S_{i,k}+S_{k,j}-(d+1) \leq l_i+l_k-A'-(k-i-1)(d+1-2l_i)+\sum_{t=k}^{j-1}(L(u_{t})+L(u_{t+1})-(d+1))+(d+1)-(d+1) = l_i+l_k-((k-i)/2)-(k-i-1)+l_k+l_j+\sum_{t=k+1}^{j-1}(2L(u_{t})-(d+1))-(d+1) = l_i+l_j+(2l_k-(d+1))-3((k-i)/2)+1+\sum_{t=k+1}^{j-1}(2L(u_t)-(d+1)) \leq l_i+l_j-1-(j-k-1)((d-2)/2)$. If $u_i$ and $u_j$ are in distinguishable legs, then $\phi(u_i,u_j) = 1$ and hence $S_{i,j} \leq l_i+l_j-2 = l_i+l_j-2\phi(u_i,u_j) = d(u_i,u_j)$. If $u_i$ and $u_j$ are in the same leg, then $\phi(u_i,u_j) = l_j$ and hence $S_{i,j} \leq l_i+l_j-1-(d-2) \leq l_i-l_j = d(u_i,u_j)$. If $u_i, u_j \not\in S(T)$, then $S_{i,j} \leq l_i+l_j-(j-i-1)(d/2)$. If $u_i$ and $u_j$ are in distinguishable legs, then $\phi(u_i,u_j) = 1$ and hence $S_{i,j} \leq l_i+l_j-2 = l_i+l_j-2\phi(u_i,u_j) = d(u_i,u_j)$. If $u_i$ and $u_j$ are in the same leg, then $\phi(u_i,u_j) = l_j$ and hence $S_{i,j} \leq l_i+l_j-2d \leq l_i-l_j = l_i+l_j-2\phi(u_i,u_j) = d(u_i,u_j)$. This completes the proof of Claim 1.

\medskip
\textsf{Case 2:} $z = 2$.

In this case, set $u_{0}=w^{2}$, $u_{1}=w_{1,h-1}^{1}$, $u_{2}=r_2$, $u_{3}=w_{2,h-1}^{1}$, $u_{4}=w_{1,h-1}^{2}$, $u_{5}=r_1$, $u_{6}=w_{2,h-1}^{2}$ and $u_{p-1}=w^{1}$. For $1 \leq l \leq 2$ and $3 \leq i \leq m$, set $u_{t} = w_{i, h-1}^{l}$, where
$$
t = 2i+l,
$$
which ranges from $7$ to $2m+2$. For $1 \leq l \leq 2$, $1 \leq i \leq m$ and $1 \leq j \leq h-2$, set $u_{t}=w_{i,j}^{l}$, where
\begin{equation*}
t =
\begin{cases}
2i+2mj+l, & \mbox{ if } l=1  \\
2i+2m(h-j-1)+l, & \mbox{ if } l=2,
\end{cases}
\end{equation*}
which ranges from $2m+3$ to $p-2$. Then $u_{0} = w_{1}$ and $u_{p-1} = w^{2}$ are adjacent to the weight centers $r_1$ and $r_2$, respectively. Note that $L(u_0)+L(u_{p-1}) = 2$. Moreover, the four vertices are immediately before and after the weight centers and all other remote vertices are ordered in an interval of even length. Thus the linear order $u_{0},u_{1},\ldots,u_{p-1}$ thus defined is admissible. Thus condition (a) in Theorem \ref{thm:lower} is satisfied. Note that, for $0 \leq t \leq p-1$, $u_{t}$ and $u_{t+1}$ are in opposite branches and hence $\phi(u_{t},u_{t+1}) = 0$ and $\delta(u_{t},u_{t+1}) = 1$. The following claim says that condition (b) in Theorem \ref{thm:lower} is also satisfied.

\medskip
\textsf{Claim 2:} The linear order $u_{0},u_{1},\ldots,u_{p-1}$ defined above satisfies condition \eqref{eqn:duv}.

Consider any two vertices $u_i,u_j$ with $0 \leq i < j \leq p-1$. It is easy to verify that condition \eqref{eqn:duv} is satisfied for $0 \leq i \leq 6$ and $j = p-1$. Hence we assume $7 \leq i < j \leq p-2$ in the sequel. Let $L(u_i) = l_i$ and $L(u_j) = l_j$. Denote $A = a_i+a_{i+1}+\ldots+a_j$. Then $l_i > l_j$ as $i < j$. The right-hand side of \eqref{eqn:duv} is equal to $S_{i,j} = \sum_{t=i}^{j-1}(L(u_t)+L(u_{t+1})-a_t)-(j-i)d+(d+1) \leq (j-i)l_i+(j-i-1)l_i+l_j-A-(j-i-1)d+1 = l_i+l_j-A-(j-i-1)(d-2l_i)+1$. If $u_i$ and $u_j$ are in opposite branches, then $d(u_i,u_j) = l_i+l_j+1$ and hence $S_{i,j} \leq l_i+l_j+1 = d(u_i,u_j)$ as $j-i \geq 1, d-2l_i \geq 1$ and $A \geq 0$. If $u_i$ and $u_j$ are in the same branch, then $d(u_i,u_j) = l_i+l_j-2\phi(u_i,u_j)$ with $\phi(u_i,u_j) \geq 1$. If $u_i,u_j \in S(T)$, then $\phi(u_i,u_j) = 1$ and hence $S_{i,j} \leq l_i+l_j-A-(j-i-1)(d-2l_i)+1 \leq l_i+l_j-(j-i)-(j-i-1)(d-2l_i)+1 \leq l_i+l_j-(j-i-1)-(j-i-1)(d-2l_i) \leq l_i+l_j-(j-i-1) \leq l_i+l_j- 2 = l_i+l_j-2\phi(u_i,u_j) = d(u_i,u_j)$. If $u_i \in S(T)$ but $u_j \not\in S(T)$, then there exists $k$ with $i \leq k < j$ such that $u_k \in S(T)$ but $u_{k+1} \not\in S(T)$. Let $L(u_k) = l_k$ and $A' = a_i+a_{i+1}+\ldots+a_k$. Then $S_{i,j} = S_{i,k}+S_{k,j}-d+1 \leq l_i+l_k-A'-(k-i-1)(d-2l_i)+1+\sum_{t=i}^{j-1}(L(u_t)+L(u_{t+1})-d)+d+1-d+1 = l_i+l_j-A'-(k-i-1)(d-2l_i)+3+ \sum_{t=k+1}^{j-1}(2L(u_t)-d)+(2l_k-d) \leq l_i+l_j-(k-i)-(k-i-1)+3+\sum_{t=k+1}^{j-1}(2L(u_{t})-d)$. If $u_i$ and $u_j$ are in distinguishable legs, then $\phi(u_i,u_j) = 1$ and hence $S_{i,j} \leq l_i+l_j-2(k-i-1)+2+\sum_{t=k+1}^{j-1}(2L(u_{t})-d) \leq l_i+l_j-2+2-2 = l_i+l_j-2 = l_i+l_j-2\phi(u_i,u_j) = d(u_i,u_j)$. If $u_i$ and $u_j$ are in the same leg, then $\phi(u_i,u_j) = l_j$ and hence $S_{i,j} \leq l_i+l_j-2(k-i)-(d-1) \leq l_i+l_j-2l_j = l_i-l_j = l_i+l_j-2\phi(u_i,u_j) = d(u_i,u_j)$. If $u_i,u_j \not\in S(T)$, then $S_{i,j} = \sum_{t=i}^{j-1}(L(u_t)+L(u_{t+1})-d)+d+1 = l_i+l_j+\sum_{t=i+1}^{j-1}(2L(u_{t})-d)-d+d+1 = l_i+l_j+\sum_{t=i+1}^{j-1}(2L(u_t)-d)+1$. If $u_i$ and $u_j$ are in distinguishable legs, then $\phi(u_i,u_j) = 1$ and hence $S_{i,j} \leq l_i+l_j-3+1 = l_i+l_j-2 = l_i+l_j-2\phi(u_i,u_j) = d(u_i,u_j)$. If $u_i$ and $u_j$ are in the same leg, then $\phi(u_i,u_j) = l_j$ and hence $S_{i,j} \leq l_i+l_j-(j-i-1)((d-1)/2)+1 \leq l_i+l_j-(d-1)+1 = l_i+l_j-(d-2) \leq l_i+l_j-2l_i = l_i+l_j-2\phi(u_i,u_j) = d(u_i,u_j)$. This completes the proof of Claim 2.

In summary, in each case above we have given a linear order of the vertices of $L^{z}_{m,h}$ which satisfies (a) and (b) in Theorem \ref{thm:lower}. Hence, by Theorem \ref{thm:lower}, $\rn(L^{z}_{m,h})$ is equal to the right-hand side of \eqref{rn:lower} which yields exactly the right-hand side of \eqref{rn:ds} in this special case.
\end{proof}

Using the linear order in the proof above, we can obtain an optimal radio labelling of $L^{z}_{m,h}$ from \eqref{f00} and \eqref{f11}. We illustrate this for $L_{3,3}^1$ and $L_{3,3}^2$ in Figure \ref{Fig22}.

\begin{figure}[ht]
\begin{center}
\includegraphics[width=5 in]{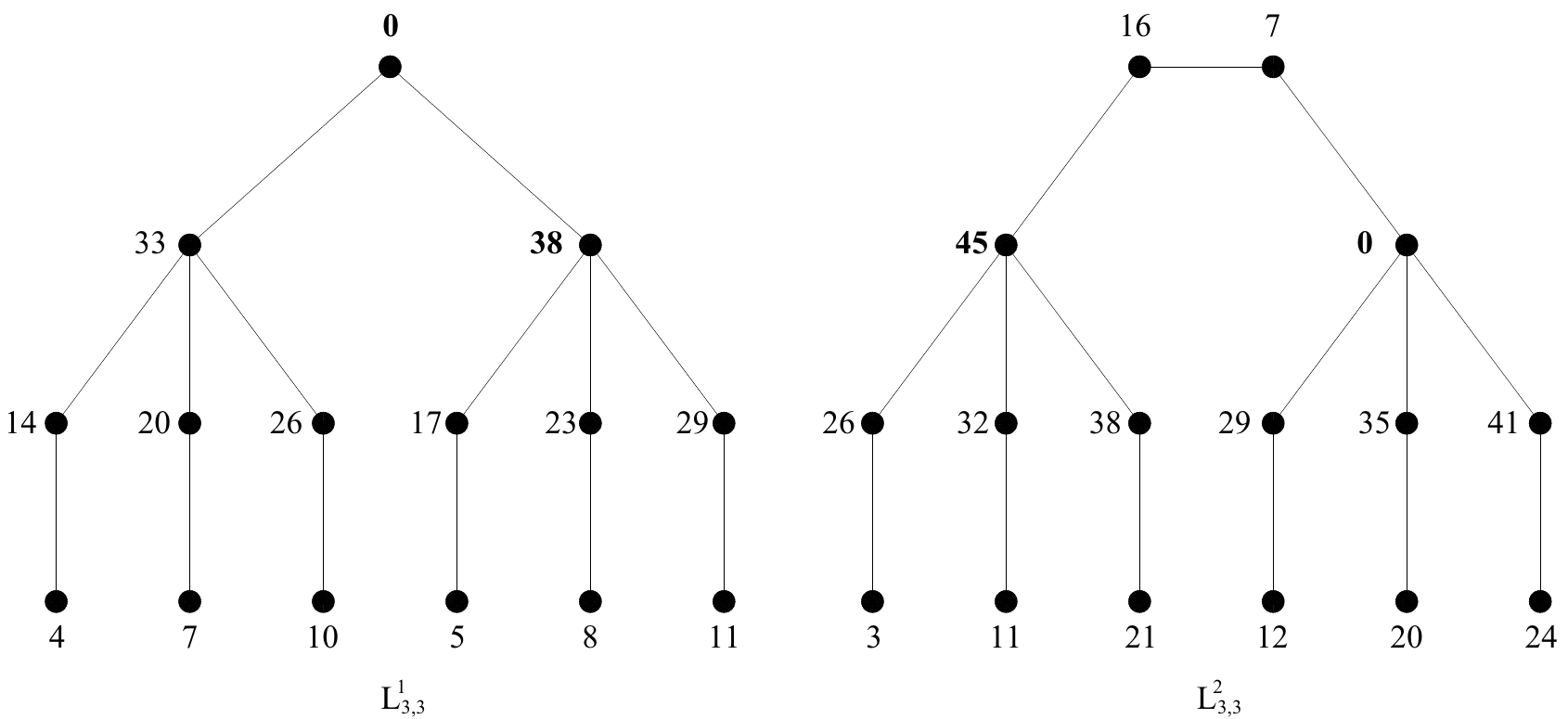}
\caption{Optimal radio labellings of $L_{3,3}^1$ and $L_{3,3}^2$ obtained from \eqref{f00} and \eqref{f11} using the linear order in the proof of Theorem \ref{thm:ds}.}
\label{Fig22}
\end{center}
\end{figure}

\section{Concluding remarks}
\label{sec:rem}

We conclude this paper with five remarks on relations between our results and two questions and three results in \cite{Daphne2}.

\begin{Remark}
A tree $T$ is called \cite{Daphne2} a \emph{lower bound tree} if $\rn(T)$ is equal to the lower bound \eqref{eqn:lb} and a \emph{non-lower bound tree} otherwise. In \cite[Question 2]{Daphne2}, Liu \textit{et al.} asked which trees with a weight center of degree two are non-lower bound trees, and moreover for a non-lower bound tree how much improvement one can make over the lower bound \eqref{eqn:lb}. Theorem \ref{thm:nlb} answers the first part of this question for a subfamily of two-branch trees, whilst Theorem \ref{thm:lower} gives an answer to the second part of this question for two-branch trees with diameter at least two together with a necessary and sufficient condition for this improved lower bound to be tight.
\end{Remark}

\begin{Remark}
In \cite[Question 1]{Daphne2}, Liu \textit{et al.} asked whether the level-wise regular trees $T^1_{m_0,m_1,\ldots,m_{h-1}}$ with $m_0=2$ and $h \geq 2$ are always non-lower bound trees, and they affirmatively answered this question in \cite[Theorem 12]{Daphne2}. It happens that this result can also be obtained from Theorem \ref{thm:nlb}. Moreover, by Theorem \ref{thm:nlb}, we also obtain that all trees $T^2_{m_0,m_1,\ldots,m_{h-1}}$ with $m_0 = 2$ and $h \geq 2$ except paths $P_{2h}$ are non-lower bound trees. Furthermore, in Theorems \ref{thm:level} and \ref{thm:ds} we give the exact value of $\rn(T^z_{m_0,m_1,\ldots,m_{h-1}})$ for $z=1,2$ when $m_0 = 2$ and $m_i \geq 3$ for $1 \leq i \leq h-1$, or $m_0 = m_2 = \cdots = m_{h-1} = 2$ and $m_1 \geq 2$, with the help of our improved lower bound in Theorem \ref{thm:lower}.
\end{Remark}

\begin{Remark}
\label{rem:comp1}
In \cite[Theorem 12]{Daphne2}, Liu, Saha and Das proved that for a single-root level-wise regular tree $T^1 = T^{1}_{2, m_1, \ldots, m_{h-1}}$ with $m_i \ge 2$ for $1 \le i \le h-1$ and $\prod_{i=1}^{k-1} (m_i - 1) \geq \frac{2k+1}{3}$ for $2 \le k \le h-1$, $\rn(T^{1}_{2, m_1, \ldots, m_{h-1}})$ attains the lower bound in \cite[Theorem 10]{Daphne2}, namely $\rn(T^1) = (2h+1)(p-1) - 2L(T^1) + 1 + |R_h(T^1)| = (p-1) (d+1) - 2L(T^1) + 1 + \prod_{i=1}^{h-1}(m_i-1)$, where $|R_h(T^1)|$ agrees with our $\xi(T^1) = \prod_{i=1}^{h-1}(m_i-1)$, $p = 3+2\sum_{i=1}^{h-1}\(\prod_{j=1}^{i}(m_j-1)\)$ is the order of $T^1$, and $d = 2h$ is the diameter of $T^1$. Since $\ve(T^1)=1$ and $L(T^1) = 2+2\sum_{i=1}^{h-1}(i+1)\(\prod_{j=1}^{i}(m_j-1)\)$, this result is identical to the first line in \eqref{rn:level} when $m_i \geq 3$. (Note that $T^{1}_{m_0, m_1, \ldots, m_{h-1}}$ in the present paper is the same as $T_{m_0, m_1-1, \ldots, m_{h-1}-1}$ but not $T_{m_0, m_1, \ldots, m_{h-1}}$ in \cite{Daphne2}.) That is, our result in Theorem \ref{thm:level} about $T^{1}_{2, m_1, \ldots, m_{h-1}}$ is the same as \cite[Theorem 12]{Daphne2}. However, double-root level-wise regular trees were not discussed in \cite{Daphne2} and our result in Theorem \ref{thm:level} about $T^{2}_{2, m_1, \ldots, m_{h-1}}$ (the second line in \eqref{rn:level}) is not covered by any result in \cite{Daphne2}. Moreover, $L^1_{m,h}$ ($= T^{1}_{2, 2, \ldots, 2}$) with $m \geq 2$ and $h \geq 2$ does not satisfy the above-mentioned condition in \cite[Theorem 12]{Daphne2}. Thus, Theorem \ref{thm:ds} covers a case which is not covered by \cite[Theorem 12]{Daphne2}, and this is a new case which attains the lower bound in \cite[Theorem 10]{Daphne2}.
\end{Remark}

In \cite{Daphne2}, Liu \textit{et al.} defined $\Omega^h_k$ to be the family of trees with height $h$, diameter $k$ and a weight center of degree 2. A tree in this family may have one weight center or two weight centers, and a tree with one weight center belongs to $\Omega^h_k$ if and only if it is a two-branch tree with height $h$ and diameter $k$. A two-branch tree with two weight centers belongs to $\Omega^h_k$ for some $h$ and $k$, but not every tree in $\Omega^h_k$ with two weight centers is a two-branch tree. Let $T \in \Omega^h_k$ and let $x \in W(T)$ be a weight center of $T$ with degree 2. Then $T - x$ consists of two branches, say, $L_T$ and $R_T$. Denote by $L_i, R_i$ the sets of vertices at distance $i$ from $x$ in $L_T, R_T$, respectively, for $1 \leq i \leq h$. The following two remarks show that the lower bounds on $\rn(T)$ proved by Liu \textit{et al.} in  \cite{Daphne2} can be stronger than, identical to, or weaker than our lower bound in \eqref{rn:lower} in different cases.

\begin{Remark}
In \cite[Theorem 10]{Daphne2}, Liu \textit{et al.} proved the following bounds for any $T \in \Omega^h_{2d}$ with order $p$ and a weight center $x$ of degree $2$: if $L_d \neq \emptyset$ and $R_d \neq \emptyset$, then
\begin{equation}
\label{eqn:Thm10a}
    \rn(T) \geq
    \begin{cases}
    (p-1)(2d+1)-2w_T(x)+\max\{|L_d|,|R_d|\}, & \mbox{ if } |L_d| \neq |R_d| \\
    (p-1)(2d+1)-2w_T(x)+1+|R_d|, & \mbox{ if } |L_d| = |R_d|;
    \end{cases}
\end{equation}
if $L_d = \emptyset$, then
  \begin{equation}
  \label{eqn:Thm10b}
    \rn(T) \geq (p-1)(2d+1)-2w_T(x)+\max\left\{\frac{1}{2}\left(\sum_{i=0}^{h-d}(2i+1)|R_{d+i}|-2\right),1\right\}.
  \end{equation}

\medskip
\textsf{Case 1:} $|W(T)| = 1$ (that is, $W(T) = \{x\}$).

In this case, we have $\ve(T)=1$, $w_T(x) = L(T)$ and $\xi(T) = \lfloor |S(T)|/2 \rfloor$. We may assume without loss of generality that $|R_d| \geq |L_d|$. Assume $L_d \neq \emptyset$ and $R_d \neq \emptyset$ first. Then $S(T) = L_d \cup R_d$ and $\xi(T) = \lfloor (|L_d|+|R_d|)/2 \rfloor$. If $|L_d| = |R_d|$, then $\xi(T) = |R_d|$ and hence the lower bounds in \eqref{rn:lower} and in the second line of \eqref{eqn:Thm10a} are identical. If $|R_d| = |L_d|+1$, then $\xi(T) = |R_d|-1$ and so the lower bounds in \eqref{rn:lower} and in the first line of \eqref{eqn:Thm10a} are identical. If $|R_d| \geq |L_d|+2$, then $\xi(T) < \max\{|L_d|, |R_d|\}$ and hence the lower bound in \eqref{rn:lower} is weaker than the one in the first line of \eqref{eqn:Thm10a}.

Now assume $L_d = \emptyset$. Then $h-d \geq 1$ and $S(T) = R_d \cup R_{d+1}$. If $h-d=1$ and $|R_{d+1}|=1$, then $\xi(T) = \lfloor (|R_d|+1)/2 \rfloor = \max\{(\sum_{i=0}^{h-d}(2i+1)|R_{d+i}|-2)/2,1\}$ and hence the lower bounds in \eqref{rn:lower} and \eqref{eqn:Thm10b} are identical. If $h-d > 1$, or $h-d = 1$ and $|R_{d+1}| > 1$, then $\xi(T) = \lfloor (\sum_{i=0}^{h-d}|R_{d+i}|)/2 \rfloor < \max\{(\sum_{i=0}^{h-d}(2i+1)|R_{d+i}|-2)/2,1\}$ and hence the lower bound in \eqref{rn:lower} is weaker than the one in \eqref{eqn:Thm10b}.

\medskip
\textsf{Case 2:} $|W(T)| = 2$ (that is, $W(T)$ consists of $x$ and another weight center).

In this case, we have $\ve(T)=0$, $w_T(x) = L(T) + (n/2)$ and  $\xi(T) = \max\{0,|S(T)|-2\}$. Assume without loss of generality that the weight center of $T$ other than $x$ is in $L_T$. Assume $L_d \neq \emptyset$ and $R_d \neq \emptyset$ first. If $|L_d| = |R_d|$, then $S(T) = R_d$ and $\xi(T) = \max\{0,|R_d|-2\} < |R_d|$; if $|L_d| \neq |R_d|$, then $1+\xi(T) = 1+\max\{0,|R_d|-2\} < \max\{|L_d|,|R_d|\}$. In either case the lower bound in \eqref{rn:lower} is weaker than the one in \eqref{eqn:Thm10a}.

Now assume $L_d = \emptyset$. Then $S(T) = \cup_{i=0}^{h-d} R_{d+i}$ and $1+\xi(T) = 1+\max\{0,\sum_{i=0}^{h-d}|R_{d+i}|-2\}$. Thus, the lower bound in \eqref{rn:lower} is stronger than, identical to, or weaker than the lower bound in \eqref{eqn:Thm10b} depending on whether $1+\max\{0,\sum_{i=0}^{h-d}|R_{d+i}|-2\}$ is greater than, equal to, or smaller than $\max\{(\sum_{i=0}^{h-d}(2i+1)|R_{d+i}|-2)/2,1\}$, respectively.

In all possibilities above, Theorem \ref{thm:lower} gives a necessary and sufficient condition for $T$ to achieve the lower bound in \eqref{rn:lower}, while no similar condition for the sharpness of \eqref{eqn:Thm10a} or \eqref{eqn:Thm10b} was given in \cite[Theorems 10]{Daphne2}.
\end{Remark}

One can see that for any $T \in \Omega^h_{2d+1}$ exactly one of the branches $L_T$ and $R_T$ contains vertices at level $d+i$ for $1 \le i \le h-d$. Without loss of generality we may assume that this branch is $R_T$.

\begin{Remark}
In \cite[Theorem 15]{Daphne2}, Liu \textit{et al.} proved that, for any $T \in \Omega^h_{2d+1}$ (where $d \geq 2$) with order $p$ and a weight center $x$ of degree 2,
\begin{equation}\label{eqn:Thm15}
\rn(T) \geq
\begin{cases}
  (p-1)(2d+2)-2w_T(x)+\max\{2|R_{d+1}|-5,1\}, & \mbox{if } h = d+1 \\
  (p-1)(2d+2)-2w_T(x)+\sum_{i=1}^{h-d}(i+1)|R_{d+i}|-2, & \mbox{if } h > d+1.
\end{cases}
\end{equation}

\medskip
\textsf{Case 1:} $|W(T)| = 1$ (that is, $W(T) = \{x\}$).

In this case, we have $\ve(T) = 1$, $w_T(x) = L(T)$ and $\xi(T) = \lfloor |S(T)|/2 \rfloor$. Assume $h = d+1$ first. Then $S(T) = R_{d+1}$ and $1+\xi(T) = 1+\lfloor |R_{d+1}|/2 \rfloor$. Thus, if $|R_{d+1}| = 2,3$, then $1+\xi(T) > \max\{2|R_{d+1}|-5,1\}$; if $|R_{d+1}| = 1, 4$, then $1+\xi(T) = \max\{2|R_{d+1}|-5,1\}$; if $|R_{d+1}| \geq 5$, then $1+\xi(T) < \max\{2|R_{d+1}|-5,1\}$. In these three possibilities the lower bound in \eqref{rn:lower} is stronger than, identical to, or weaker than the lower bound in \eqref{eqn:Thm15}, respectively.

Now assume $h > d+1$. Then $S(T) = \cup_{i=1}^{h-d}R_{d+i}$ and hence $1+\xi(T) = 1+\lfloor (\sum_{i=1}^{h-d}|R_{d+i}|)/2 \rfloor < \sum_{i=1}^{h-d}(i+1)|R_{d+i}|-2$. So the lower bound in \eqref{rn:lower} is weaker than the one in \eqref{eqn:Thm15}.

\medskip
\textsf{Case 2:} $|W(T)| = 2$ (that is, $W(T)$ consists of $x$ and another weight center).

In this case, we have $\ve(T) = 0$, $w_T(x) = L(T) + (n/2)$ and $\xi(T) = \max\{0,|S(T)|-2\}$. Assume $h=d+1$ first. Assume without loss of generality that the weight center of $T$ other than $x$ is in $R_T$. Then $S(T) = L_d \cup R_{d+1}$ with $|L_d| \neq \emptyset$ and $|R_{d+1}| \neq \emptyset$. If $|L_d| = |R_{d+1}|$, then $1+\xi(T) = 1+\max\{0, |L_d|+|R_{d+1}|-2\} = 2|R_{d+1}|-1$, and hence the lower bound in \eqref{rn:lower} is stronger than or identical to the lower bound in the first line of \eqref{eqn:Thm15} according to whether $|R_{d+1}| \ge 2$ or $|R_{d+1}|=1$, respectively. Similarly, the lower bound in \eqref{rn:lower} is stronger than, identical to, or weaker than the lower bound in the first line of \eqref{eqn:Thm15} according to whether $|R_{d+1}|-|L_d| \leq 5$, $|R_{d+1}| = |L_d|+6$, or $|R_{d+1}| > |L_d|+6$, respectively.

Now assume $h > d+1$. Then $S(T) = \cup_{i=1}^{h-d} R_{d+i}$ and $\xi(T) = \sum_{i=1}^{h-d}|R_{d+i}| - 2 < \sum_{i=1}^{h-d}(i+1)|R_{d+i}|-3$, and hence the lower bound in \eqref{rn:lower} is weaker than the one in the second line of \eqref{eqn:Thm15}.

In all possibilities above, Theorem \ref{thm:lower} has the benefit that it gives a necessary and sufficient condition for $T$ to achieve the lower bound \eqref{rn:lower}, while no similar condition for the sharpness of \eqref{eqn:Thm15} was given in \cite[Theorems 15]{Daphne2}.

As an example to show that \eqref{rn:lower} can outperform \eqref{eqn:Thm15} sometimes, let us consider $T = C(n,k)$ with $n > 4$ even and $k \geq 2$ (see Theorem \ref{thm:cat}). For this family of trees, we have $|W(T)| = 2$, $|L_d| = |R_{d+1}| \ge 2$ and $1+\xi(T) = 2k-1 = 2|R_{d+1}|-1 > \max\{2|R_{d+1}|-5,1\}$, and hence the lower bound in \eqref{rn:lower} (which is tight and is equal to the fourth line in \eqref{rn:cat}) is bigger than the lower bound in the first line of \eqref{eqn:Thm15} by $4$.
\end{Remark}

\section*{Acknowledgment} 
The authors are grateful to the two anonymous referees for their careful reading of the manuscript and helpful comments.

\end{document}